\documentclass[a4paper,11pt]{article}
\usepackage{amsthm,amsmath,amsfonts,amssymb,hyperref}
\usepackage{microtype}
\usepackage{graphicx,color}
\usepackage{calrsfs}
\numberwithin{equation}{section}
\newtheorem{thm}{Theorem}[section]
\newtheorem{pro}{Proposition}[section]
\newtheorem{cor}{Corollary}[section]
\newtheorem{lmm}{Lemma}[section]
\theoremstyle{definition}
\newtheorem{dfn}{Definition}[section]
\newtheorem{ass}{Assumption}[section]
\newtheorem{rem}{Remark}[section]

\renewcommand{\d}{\mathrm{d}}
\newcommand{\e}{\mathrm{e}}

\newcommand{\R}{\mathbb{R}}

\newcommand{\E}{\mathbf{E}}
\renewcommand{\P}{\mathbf{P}}

\renewcommand{\L}{\mathcal{L}}

\begin{document}
\title{Parameter Estimation for the Langevin Equation with Stationary-Increment Gaussian Noise}
\author{Tommi Sottinen \and Lauri Viitasaari}

\renewcommand{\thefootnote}{\fnsymbol{footnote}}

\author{Tommi Sottinen\footnotemark[1] \, and \, Lauri Viitasaari\footnotemark[2]}

\footnotetext[1]{Department of Mathematics and Statistics, University of Vaasa, P.O.Box 700, FIN-65101 Vaasa, Finland, \texttt{tommi.sottinen@iki.fi}.

T. Sottinen was partially funded by the Finnish Cultural Foundation (National Foundations' Professor Pool).}

\footnotetext[2]{Department of Mathematics and System Analysis, Aalto University School of Science, Helsinki P.O. Box 11100, FIN-00076 Aalto, Finland, \texttt{lauri.viitasaari@aalto.fi}

L.Viitasaari was partially funded by the Emil Aaltonen Foundation.}

\maketitle

\begin{abstract}
\noindent
We study the Langevin equation with stationary-increment Gaussian noise.  We show the strong consistency and the asymptotic normality with Berry--Esseen bound of the so-called alternative estimator of the mean reversion parameter. 
%The strong consistency follows from the ergodicity of the stationary solution, while the asymptotic normality and Berry-Esseen bound follow from the fourth moment theorem. 
The conditions and results are stated in terms of the variance function of the noise.  We consider both the case of continuous and discrete observations. 
As examples we consider fractional and bifractional Ornstein--Uhlenbeck processes.
Finally, we discuss the maximum likelihood  and the least squares estimators.
\end{abstract}

{\small
\medskip


\medskip

}

%%%%%%%%%%%%%%%%%%%%%%%%%%%%%%%%%%%%%%%%%%%%%%%%%%%%%%%%%%%%%%%%%%%%%%%%%%%%%%%

%\tableofcontents

%%%%%%%%%%%%%%%%%%%%%%%%%%%%%%%%%%%%%%%%%%%%%%%%%%%%%%%%%%%%%%%%%%%%%%%%%%%%%%%
\section{Introduction}

We consider statistical parameter estimation for the unknown parameter $\theta>0$ in the (generalized) Langevin equation 
\begin{equation}\label{eq:langevin}
\d U^{\theta,\xi}_t = -\theta U^{\theta,\xi}_t \,\d t + \d G_t, \qquad t\ge 0.
\end{equation}
Here the noise $G$ is Gaussian, centered, and has stationary increments. The initial condition $U^{\theta,\xi}_0=\xi$ can be any centered Gaussian random variable.
We consider the so-called Alternative Estimator (AE) and show its strong consistency and asymptotic normality, and provide Berry--Esseen bound for the normal approximation.  
The AE was named thus by Hu and Nualart \cite{Hu-Nualart-2010}.  A more apt name for the estimator could be the Ergodic Estimator, as it uses the ergodicity of the solution directly.  We kept the name AE, though.

The Langevin equation is named thus by the pioneering work of Langevin \cite{Langevin-1908}. Sometimes the solutions to the Langevin equation are called Ornstein--Uhlenbeck processes, due to the pioneering work of Ornstein and Uhlenbeck \cite{Ornstein-Uhlenbeck-1930}. In these works the noise was the Brownian motion, and in this case the equation has been studied extensively since; see, e.g., Liptser and Shiryaev \cite{Liptser-Shiryaev-II-2001} and the references therein.
Recently, the Langevin equation with fractional Brownian noise, i.e., the fractional Ornstein--Uhlenbeck processes, have been studied extensively in, e.g., \cite{Azmoodeh-Viitasaari-2015a,Es-Sebaiy-Ndiaye-2014,Es-Sebaiy-Tudor-2015,Kozachenko-Melnikov-Mishura-2015,Kubilius-Mishura-Ralchenko-Seleznjev-2015,Shen-Yin-Yan-2016,Shen-Xu-2014,Sun-Guo-2015,Tanaka-2015}, just to mention a few very recent ones.  

The rest of the paper is organized as follows: In Section \ref{sect:prelim} we consider the Langevin equation is a general setting and provide some general results.
Section \ref{sect:ae} is the main section of the paper.  There we introduce the AE and provide assumptions ensuring its strong consistency and asymptotic normality, or the central limit theorem.  We also provide Berry--Esseen bounds for the central limit theorem, and consider the estimation based on discrete observations.   
In Section \ref{sect:examples} we provide examples. We show how some recent results concerning the fractional Ornstein--Uhlenbeck processes follow in a straightforward manner from our results, and extend the previous results.  We also study the bifractional Ornstein--Uhlenbeck processes of the second kind.
In Section \ref{sect:discussion} we discuss Least Squares Estimators (LSE) and Maximum Likelihood Estimators (MLE).  We argue that the AE is, under the ergodic hypothesis, the most general estimator one could hope for.  Moreover, we argue that the LSE is not appropriate in many cases.  For the MLE, we point out how it could be used in the general Gaussian setting.
In Section \ref{sect:conclusions} we draw some conclusions.
Finally, the proofs of all the lemmas of the paper are given in Appendix \ref{apx}.

%%%%%%%%%%%%%%%%%%%%%%%%%%%%%%%%%%%%%%%%%%%%%%%%%%%%%%%%%%%%%%%%%%%%%%%%%%%%%%%
\section{Preliminaries}\label{sect:prelim}

\subsection{General Setting}

Let us first consider the Langevin equation \eqref{eq:langevin} in a general setting, where $G$ is simply a stochastic process, and the initial condition $\xi$ is any random variable. The solution of \eqref{eq:langevin} is
\begin{equation}\label{eq:U-nonstat}
U^{\theta,\xi}_t = \e^{-\theta t}\xi + \int_0^t \e^{-\theta(t-s)}\, \d G_s.
\end{equation}
Indeed, nothing is needed here, except the finiteness of the noise: \eqref{eq:U-nonstat} is the unique solution to \eqref{eq:langevin} in the pathwise sense, and the stochastic integral in \eqref{eq:U-nonstat} can be defined pathwise by using the integration by parts as
$$
\int_0^t \e^{-\theta(t-s)}\, \d G_s 
=
G_t -\theta\int_0^t \e^{-\theta(t-s)} G_s\, \d s. 
$$
Any two solutions $U^{\theta,\xi}$ and $U^{\theta,\zeta}$ with the same noise are connected by the relation
$$
U^{\theta,\zeta}_t = U^{\theta,\xi}_t + \e^{-\theta t}\big(\zeta-\xi\big).
$$
Since our estimation is be based on the solution that starts from zero, we introduce the notation $X^\theta = U^{\theta,0}$.

For the existence of the stationary solution, the noise $G$ must have stationary increments.  Then, by extending $G$ to the negative half-line with an independent copy running backwards in time, the stationary solution is
\begin{equation}\label{eq:U}
U^\theta_t = \int_{-\infty}^t \e^{-\theta(t-s)}\, \d G_s, \quad t\ge 0.
\end{equation} 
In other words, the stationary solution is $U^\theta=U^{\theta,\xi_{\mathrm{stat}}}$, with
$$
\xi_{\mathrm{stat}} = \int_{-\infty}^0 \e^{-\theta t}\, \d G_t.
$$
In particular, the stationary solution exists if and only if the integral above converges (almost surely), and in this case
\begin{equation}\label{eq:X}
X^\theta_t = U^\theta_t - \e^{-\theta t} U^\theta_0.
\end{equation}

\begin{rem}
By \cite[Theorem 3.1]{Viitasaari-2014-preprint} all stationary processes are the stationary solutions \eqref{eq:U} of \eqref{eq:langevin} with suitable stationary-increment noise $G$ and parameter $\theta$.  Also, by Barndorff and Basse-O'Connor \cite[Theorem 2.1]{Barndorff-Nielsen-Basse-OConnor-2011} the stationary solution of \eqref{eq:langevin} exists for all integrable stationary-increment noises.
\end{rem}

\subsection{Second Order Stationary-Increment Setting}

Assume that the noise $G$ is centered square-integrable process with stationary increments.  

\begin{rem}[Some notation]

By $v$ we denote variance of $G$, by $r_\theta$ the autocovariance of $U^\theta$, and by $\gamma_\theta$ the covariance of $X^\theta$. 
By $\Phi$ and $\bar\Phi$ we denote the cumulative and complementary cumulative distribution functions of the $\mathcal{N}(0,1)$-distributed variable, respectively; $\mathcal{N}(0,1)$ denotes the standard normal distribution. 
By $C$ we will denote a universal constant depending only on $v$;  $C_\theta$ and $C_{\theta,K}$, and so on, are universal constants depending additionally on $\theta$, and $\theta$ and $K$, and so on. In proofs, the constants may change from line to line and sometimes the dependence on the parameters are suppressed. 
We use the asymptotic notation $f(T) \sim g(T)$ for 
$$
\lim_{T\to\infty} \frac{f(T)}{g(T)} = 1.
$$
\end{rem}

The existence of the stationary covariance $r_\theta$, given by Proposition \ref{pro:r}, is ensured by the following elementary lemma.

\begin{lmm}\label{lmm:v-growth}
Let $v\colon\mathbb{R}\to\mathbb{R}$ be a variance function of a process having stationary increments. Then, for all $t>0$,  
$$
v(t) \le C \,t^2.   
$$
\end{lmm}

\begin{pro}\label{pro:r}
\begin{equation*}
r_\theta (t) = \theta \int_{-\infty}^0\e^{\theta s} g(t,s)\, \d s
-\theta^2\e^{-\theta t} \int_{-\infty}^t\!\int_{-\infty}^0 \e^{\theta(s+u)}g(s,u)\, \d s\d u,
\end{equation*}  
where
$$
g(t,s) = \frac{1}{2}\Big[ v(t) + v(s) - v(t-s)\Big].
$$
In particular,
\begin{equation}\label{eq:rzero}
r_\theta(0) = \frac{\theta}{2}\int_0^\infty \e^{-\theta t}v(t)\, \d t.
\end{equation}
\end{pro}

\begin{proof}%[Proof of Proposition \ref{pro:r}]
By integrating by parts, we obtain
$$
r_\theta(t) 
=
\E\left[\int_{-\infty}^0 \theta\e^{\theta s} G_t G_s\, \d s
-\int_{-\infty}^t\int_{-\infty}^0 \theta^2\e^{-\theta(t-s-u)} G_s G_u\, \d s \d u\right]. 
$$
The claim follows from this by the Fubini's theorem, if the integrals above converge. To this end, it is necessary and sufficient that $r_\theta(0)$ is finite.  Now,
\begin{eqnarray*}
r_\theta(0) &=& 
%\theta^2 \int_{-\infty}^0\!\int_{-\infty}^0 \e^{\theta (t+s)}g(t,s)\, \d t\d s\\
%&=& 
\frac{\theta^2}{2} \int_{-\infty}^0\!\int_{-\infty}^0 \e^{\theta t}\e^{\theta s}\Big[v(t)+v(s)-v(t-s)\Big]\, \d t\d s \\
&=& \theta\int_{-\infty}^0 \e^{\theta t}v(t)\, \d t - \frac{\theta^2}{2} \int_{-\infty}^0 \e^{\theta t}\left[\int_{-\infty}^{-t} \e^{\theta(t+s)}v(s)\, \d s\right]\d t.
\end{eqnarray*}
For the latter term we have
\begin{eqnarray*}
\lefteqn{\frac{\theta^2}{2} \int_{-\infty}^0 \e^{\theta t}\left[\int_{-\infty}^{-t} \e^{\theta(t+s)}\, \d s\right]\d t} \\
&=& 
\frac{\theta^2}{2} \int_{-\infty}^{\infty} v(s)\e^{\theta s}\left[\int_{-\infty}^{\min(-s,0)} \e^{2\theta t}\, \d t\right]\d s \\
&=&
\frac{\theta^2}{2}\left[\int_{-\infty}^0 v(s)\e^{\theta s}\left[\int_{-\infty}^{0} \e^{2\theta t}\, \d t\right]\d s
+ \int_{0}^{\infty} v(s)\e^{\theta s}\left[\int_{-\infty}^{-t} \e^{2\theta t}\, \d t\right]\d s\right] \\
&=&
\frac{\theta}{4}\left[\int_{-\infty}^0 v(s) \e^{\theta s}\, \d s 
+\int_0^\infty v(s)\e^{\theta s} \e^{-2\theta s}\, \d s \right] \\
&=&
\frac{\theta}{2}\int_0^\infty e^{-\theta s} v(s)\, \d s.
\end{eqnarray*}
Consequently, we have shown \eqref{eq:rzero}. Since, by Lemma \ref{lmm:v-growth}, $v(t)\le C t^2$, the finiteness of $r_\theta$ follows from the representation above.
\end{proof}

%Proposition \ref{pro:gamma-vs-r} below states that, in the $L^2$-sense, the zero-initial solution converges exponentially fast to the stationary solution.

\begin{pro}\label{pro:gamma-vs-r}
$$
\gamma_\theta(t,s) = r_\theta(t-s) + \e^{-\theta(t+s)}r_\theta(0) - \e^{-\theta t}r_\theta(s) - \e^{-\theta s}r_\theta(t). 
$$
In particular,
$$
k_\theta(t,s) =
\big|\gamma_\theta(t,s)-r_\theta(t-s)\big| 
\le C_\theta \e^{-\theta\min(t,s)}.  % C DOES NOT DEPEND ON THETA?
$$
\end{pro}

\begin{proof}%[Proof of Proposition \ref{pro:gamma-vs-r}]
The formula for $\gamma_\theta$ is immediate from \eqref{eq:X}. As for the estimate, note that $|r_\theta(t)|\le r_\theta(0)$ by the Cauchy--Schwarz inequality.  Consequently, assuming $s\le t$,
\begin{eqnarray*}
k_\theta(t,s) 
&\le& \e^{-\theta(t+s)}r_\theta(0) + \e^{-\theta t}r_\theta(0) + \e^{-\theta s}r_\theta(0) \\
&=& r_\theta(0)\left[ \e^{-\theta t} + \e^{-\theta(t-s)} + 1\right] \e^{-\theta s},
\end{eqnarray*}
from which the estimate follows.
\end{proof}

\subsection{Gaussian Setting}

Assume that the stationary-increment noise $G$ in the Langevin equation \eqref{eq:langevin} is centered, continuous and Gaussian. Then (and only then) the continuous stationary Gaussian solution exists and can be characterized by its autocovariance function $r_\theta$ given by Proposition \ref{pro:r}.

\begin{rem}[Continuity]
In the Gaussian realm the assumption that $G$ is continuous is essential.  Indeed, if $G$ were discontinuous at any point, then $U^\theta$ would be discontinuous at every point, and also unbounded on every interval by the Belyaev's alternative \cite{Belyaev-1960}.  Parameter estimation for such a $U^\theta$ would be a fools errand, indeed.
\end{rem}

\begin{rem}[Gaussian assumption]
The assumption of Gaussianity is not needed in construction of the AE in Definition \ref{dfn:ae}.  Also, the strong consistency result of Theorem \ref{thm:consistency} does not rely on Gaussianity.  However, Assumption \ref{ass:ergodic} expresses ergodicity in terms of the autocovariance function $r_\theta$ and this is essentially a Gaussian characterization.  Theorem \ref{thm:consistency} will remain true for any square-integrable continuous stationary-increment centered noise once Assumption \ref{ass:ergodic} is replaced by a suitable assumption that ensures the ergodicity of the stationary solution.  On the contrary, the proof of Theorem \ref{thm:normality} concerning the asymptotic normality of the AE relies heavily on the assumption of Gaussianity, and cannot be generalized in any straightforward manner to non-Gaussian noises.  
\end{rem}

%%%%%%%%%%%%%%%%%%%%%%%%%%%%%%%%%%%%%%%%%%%%%%%%%%%%%%%%%%%%%%%%%%%%%%%%%%%%%%%
\section{Alternative Estimator}\label{sect:ae}

For the AE of Definition \ref{dfn:ae} below to be well-defined we need the invertibility of $\psi(\theta) = r_\theta(0)$, which is ensured by the following assumption:

\begin{ass}[Invertibility]\label{ass:invertibility}
$v$ is strictly increasing.
\end{ass}

\begin{lmm}[Invertibility]\label{lmm:invertibility}
Suppose Assumption \ref{ass:invertibility} holds.  Then $\psi\colon\R_+\to(0,\psi(0+))$ is strictly decreasing infinitely differentiable convex bijection.
\end{lmm}

\begin{dfn}[AE]\label{dfn:ae}
The alternative estimator is
$$
\tilde\theta_T = \psi^{-1}\left(\frac{1}{T}\int_0^T (X^\theta_t)^2\, \d t\right),
$$
where 
$$
\psi(\theta) = \frac{\theta}{2}\int_0^\infty \e^{-\theta t} v(t)\, \d t
$$
is the variance of the stationary solution.
\end{dfn}

\begin{rem}
The idea of the AE is to use the ergodicity of the stationary solution directly.  Therefore, it would have been more natural to base it on the stationary solution $U^\theta$ instead of the zero-initial solution $X^\theta$. However, from the practical point of view, using the solution $X^\theta$ makes more sense, since it does not assume that the Ornstein--Uhlenbeck process has reached its stationary state. Moreover, the use of the zero-initial solution instead of the stationary solution makes no difference (except when bias is concerned; see Remark \ref{rem:bias}).  Indeed, by virtue of Proposition \ref{pro:U-vs-X} below, we could have used any solution $U^{\theta,\xi}$ with any initial condition $\xi$.
\end{rem}

\begin{pro}\label{pro:U-vs-X}
Suppose the stationary solution $U^\theta$ is ergodic.  Then, for all initial distributions $\xi$  
$$
\lim_{T\to\infty}\frac{1}{T}\int_0^T (U^{\theta,\xi}_t)^2\, \d t = \psi(\theta)
\quad \mbox{a.s.}
$$
\end{pro}

\begin{proof}%[Proof of Proposition \ref{proposition:U-vs-X}]
Let us write
\begin{eqnarray*}
\lefteqn{\frac{1}{T}\int_0^T (U^{\theta,\xi}_t)^2\, \d t} \\
&=&
\frac{1}{T}\int_0^T(U^{\theta}_t + \e^{-\theta t}(\xi-U_0^\theta))^2\, \d t \\
&=&
\frac{1}{T}\int_0^T (U^{\theta}_t)^2\, \d t 
+ \frac{2(\xi-U_0^\theta)}{T}\int_0^T \e^{-\theta t}U^{\theta}_t\, \d t
+ \frac{(\xi-U^\theta_0)^2}{T}\int_0^T \e^{-2\theta t}\, \d t.
\end{eqnarray*}
By ergodicity, the first term converges to $\psi(\theta)$ almost surely.  Also, it is clear that the third term converges to zero almost surely.  As for the second term, note that $U^\theta$ is ergodic and centered, which implies that
$$
\frac{1}{T}\int_0^T U^\theta_t\, \d t \to 0 \quad{a.s.}.
$$
Consequently, the second term converges to zero almost surely.
\end{proof}

%%%%%%%%%%%%%%%%%%%%%%%%%%%%%%%%%%%%%%%%%%%%%%%%%%%%%%%%%%%%%%%%%%%%%%%%%%%%%%%
\subsection{Strong Consistency}

The strong consistency of the AE will follow directly from the ergodicity. For Gaussian processes, the necessary and sufficient conditions for ergodicity are well known and date back to Grenander \cite{Grenander-1950} and Maruyama \cite{Maruyama-1949}. We use the following characterization for ergodicity:

\begin{ass}[Ergodicity]\label{ass:ergodic}
The autocovariance $r_\theta$ satisfies
$$
\lim_{T\to\infty} \frac{1}{T} \int_0^T |r_\theta(t)|\, \d t = 0.
$$
\end{ass}

\begin{rem}[Gaussian Ergodicity]\label{rem:ergodic}
In addition to Assumption \ref{ass:ergodic}, other well-known equivalent characterizations for the ergodicity in the Gaussian realm are
\begin{enumerate}
\item
$$
\lim_{T\to\infty} \frac{1}{T} \int_0^T r_\theta(t)^2\, \d t = 0.
$$ 
\item
The spectral measure $\mu_\theta$ defined by the Bochner's theorem
$$
r_\theta(t) = \int_{-\infty}^\infty \e^{-\mathrm{i}\lambda t} \, \mu_\theta(\d\lambda)
$$
has no atoms.
\end{enumerate}
\end{rem}

\begin{thm}[Strong Consistency]\label{thm:consistency}
Suppose Assumption \ref{ass:ergodic} and Assumption \ref{ass:invertibility} hold. Then
$$
\tilde\theta_T \to \theta
$$
almost surely as $T\to\infty$.
\end{thm}

\begin{proof}
By Assumption \ref{ass:ergodic}, the stationary solution $U^\theta$ is ergodic. Consequently, by Proposition \ref{pro:U-vs-X}
$$
\lim_{T\to\infty}\frac{1}{T}\int_0^T (X^\theta_t)^2\, \d t = \psi(\theta).
$$
Since, by Lemma \ref{lmm:invertibility}, $\psi$ is a continuous bijection, the claim follows from the continuous mapping theorem.
\end{proof}

\begin{rem}[Bias]\label{rem:bias}
Unbiasedness is a fragile property, as it is not preserved in non-linear transformations. Thus, it is not surprising that the AE is biased. Indeed, suppose we use the stationary solution $U^\theta$ instead of $X^\theta$ in the AE. Let us call this Stationary Alternative Estimator  (SAE), and denote it by $\ddot\theta_T$. Then
$$
\E[\psi(\ddot\theta_T)] 
=
\frac{1}{T} \int_0^T \E[(U^\theta_t)^2]\, \d t
= 
\frac{1}{T}\int_0^T \psi(\theta)\, \d t
= \psi(\theta).
$$
So, the SAE is unbiased for $\psi(\theta)$. However, $\psi$ is strictly convex, with makes $\psi^{-1}$ strictly concave. Consequently,  $\E[\ddot\theta_T] < \theta$. For the estimation based on the zero-initial solution $X^\theta$, even $\psi(\tilde\theta_T)$ is biased, but asymptotically unbiased.  Indeed, straightforward calculation shows that
$$
\E[\psi(\tilde\theta_T)]
=
\psi(\theta) + \frac{1}{T}\left[\int_0^T \e^{-\theta t} r_\theta(t)\, \d t + \frac{\psi(\theta)}{2\theta}\left[1-\e^{-2\theta T}\right]\right].
$$
In principle, since the distribution of $\tilde\theta_T$ and the function $\psi$ are known, it is possible to construct an unbiased alternative estimator.  However, the formula would be very complicated and, moreover, it would depend on the unknown parameter $\theta$.
\end{rem}

%%%%%%%%%%%%%%%%%%%%%%%%%%%%%%%%%%%%%%%%%%%%%%%%%%%%%%%%%%%%%%%%%%%%%%%%%%%%%%%
\subsection{Asymptotic Normality}

It turns out that the rate of convergence and the corresponding Berry--Esseen bound for the AE are given by
\begin{eqnarray*}
w_\theta(T) &=& \frac{2}{T^2}\int_0^T\!\!\!\!\int_0^T r_\theta(t-s)^2\, \d s\d t, \\
R_{\theta}(T) &=&
\frac{\int_0^T |r_\theta(t)|\, \d t}{T\sqrt{w_\theta(T)}}.
\end{eqnarray*}
This leads to the following assumption for the asymptotic normality:

\begin{ass}[Normality]\label{ass:normal}
$R_\theta(T)\to 0$ as $T\to\infty$.
\end{ass}

Our main result, Theorem \ref{thm:normality} below, shows that the AE satisfies asymptotic normality with asymptotic variance $w_\theta(T)/\psi'(\theta)^2$ and the Berry--Esseen bound for the normal approximation is governed by $R_\theta(T)$.

\begin{thm}[Asymptotic Normality with Berry--Esseen Bound]\label{thm:normality}
Suppose Assumption \ref{ass:ergodic} and Assumption \ref{ass:invertibility} hold. Then for all $K>0$ there exists a constant $C_{\theta,K}$ such that
$$
\sup_{|x|\le K} \left|\P\left[\frac{|\psi'(\theta)|}{\sqrt{w_\theta(T)}}\left(\tilde\theta_T-\theta\right)\le x\right]-\Phi(x)\right| \le C_{\theta,K} R_\theta(T).
$$ 
In particular, if Assumption \ref{ass:normal} holds, then
$$
\frac{|\psi'(\theta)|}{\sqrt{w_\theta(T)}}\left(\tilde\theta_T-\theta\right)
\stackrel{\mathrm{d}}{\to} \mathcal{N}(0,1).
$$
\end{thm}

The proof of Theorem \ref{thm:normality} uses the fourth moment Berry--Esseen bound due to Peccati and Taqqu \cite[Theorem 11.4.3]{Peccati-Taqqu-2011} that is stated below as Proposition \ref{pro:fourth-berry-esseen}.
The setting of Proposition \ref{pro:fourth-berry-esseen} is as follows: Let $W=(W_t)_{t\in\R_+}$ be the Brownian motion, and let $\P_W$ be its distribution on $L^2(\mathbb{R}_+)$. The $q^{\mathrm{th}}$ Wiener chaos is the closed linear subspace of $L^2(\Omega,\mathcal{F}_W,\P_W)$ generated by the random variables $H_q(\xi)$, where $H_q$ is the $q^{\mathrm{th}}$ Hermite polynomial
$$
H_q(x) = \frac{(-1)^q}{q!}\e^{x^2/2}\frac{\d^q}{\d x^q}\left[\e^{-x^2/2}\right],
$$  
and $\xi = \int_0^\infty f(t)\, \d W_t$ for some $f\in L^2(\mathbb{R}_+)$.

\begin{pro}[Fourth Moment Berry--Esseen Bound]\label{pro:fourth-berry-esseen}
Let $F$ belong to the $q^{\mathrm{th}}$ Wiener chaos with some $q\ge 2$. Suppose $\E[F^2]= 1$.  Then 
$$
\sup_{x\in\mathbb{R}} \Big|\P[F\le x]-\Phi(x)\Big|
\le 2\sqrt{\frac{q-1}{3q}}\, \sqrt{\E[F^4]-3}.
$$
\end{pro}

The following series of elementary lemmas deal with Gaussian processes in general, not the Gaussian solutions to the Langevin equation in particular.  To emphasize this, we drop the parameter $\theta$ in the notation. 
In this general setting, $X=(X_t)_{t\in\R_+}$ is a centered Gaussian process with continuous covariance function $\gamma\colon\R_+^2\to\R$ and 
$$
Q_T = \frac{1}{T}\int_0^T \Big[X_t^2 -\E[X_t^2]\Big]\, \d t
$$

\begin{lmm}\label{lmm:second-chaos}
$Q_T$ belongs to the $2^{\mathrm{nd}}$ Wiener chaos.  
\end{lmm}

\begin{lmm}\label{lmm:moments}
\begin{eqnarray*}
\E\left[Q_T^2\right] &=& \frac{2}{T^2}\int_{[0,T]^2} \gamma(t_1,t_2)^2\, \d t_1\d t_2 \\
\E\left[Q_T^4\right] &=&
12\left[\frac{1}{T^2}\int_{[0,T]^2} \gamma(t_1,t_2)^2\, \d t_1\d t_2\right]^2 \\
& & + \frac{24}{T^4}\int_{[0,T]^4} \gamma(t_1,t_2)\gamma(t_2,t_3)\gamma(t_3,t_4)\gamma(t_4,t_1)\, \d t_1 \d t_2 \d t_3 \d t_4.
\end{eqnarray*} 
\end{lmm}  

\begin{lmm}\label{lmm:inequality}
All continuous covariance functions $\gamma$ satisfy 
\begin{eqnarray*}
\lefteqn{\int_{[0,T]^4} \gamma(t_1,t_2)\gamma(t_2,t_3)\gamma(t_3,t_4)\gamma(t_4,t_1)\, \d t_1\d t_2 \d t_3 \d t_4 }\\
&\le&
\left[\sup_{t\in [0,T]} \int_0^T |\gamma(t,t_1)|\, \d t_1\right]^2 \int_{[0,T]^2} \gamma(t_1,t_2)^2\, \d t_1\d t_2.
\end{eqnarray*}
\end{lmm} 

\begin{lmm}\label{lmm:Q-fourth}
There exists a constant $C$ such that
\begin{eqnarray*}
\sup_{x\in\R}\left|\P\left[\frac{Q_T}{\sqrt{\E[Q_T^2]}}\,\le\, x\right]-\Phi(x)\right|
&\le&
C\,\frac{\sup_{t\in[0,T]}\int_0^T |\gamma(t,s)|\, \d s}{\sqrt{\int_0^T\!\!\int_0^T \gamma(t,s)^2\, \d t \d s}}.
\end{eqnarray*}
\end{lmm}

Let us then turn back to the special case of the Langevin equation. To this end, we decompose
$$
\frac{1}{T} \int_0^T (X^\theta_t)^2\, \d t - \psi(\theta)
= Q^\theta_T + \varepsilon_\theta(T),
$$
where
\begin{eqnarray*}
Q^\theta_T &=& \frac{1}{T}\int_0^T \Big[ (X^\theta_t)^2 - \E[(X^\theta_t)^2] \Big]\, \d t,  \\
\varepsilon_\theta(T) &=& \frac{1}{T}\int_0^T \Big[ \E[(X^\theta_t)^2] - \psi(\theta)\Big]\, \d t. 
\end{eqnarray*}
Now, the quadratic functional $Q^\theta_T$ belongs to the $2^{\mathrm{nd}}$ Wiener chaos, and the idea is to show that $Q^\theta_T$ converges to a Gaussian limit with asymptotic variance $w_\theta(T)/\psi'(\theta)^2$ and the associated Berry--Esseen bound $C_{\theta}R_\theta(T)$, while the remainder $\varepsilon_\theta(T)$ is negligible. 

\begin{lmm}[Asymptotic Variance]\label{lmm:asy-var}
$$
\left|\frac{\E[(Q^\theta_T)^2]}{w_\theta(T)}-1\right|
\le \frac{C_\theta}{T w_\theta(T)},
$$
\end{lmm}

\begin{lmm}[Equivalence of Variance]\label{lmm:equiv-var}
In general, 
$$
\E[(Q^\theta_T)^2] \sim w_\theta(T) = \frac{4}{T}\int_0^T r_\theta(t)^2(T-t)\, \d t.
$$
In particular, if $\int_0^\infty r_\theta(t)^2\,\d t<\infty$, we obtain the best possible asymptotic rate 
$$
\E[(Q^\theta_T)^2] \sim \frac{4\int_0^\infty r_\theta(t)^2\, \d t}{T}.
$$
\end{lmm}

\begin{lmm}[Berry--Esseen Bound]\label{lmm:Q-conv}
There exists a constant $C_\theta$ such that
$$
\sup_{x\in\R}\left|\P\left[\frac{Q^\theta_T}{\sqrt{w_\theta(T)}}\le x\right]-\Phi(x)\right|
\le
C_\theta R_\theta(T).
$$
\end{lmm}

\begin{proof}[Proof of Theorem \ref{thm:normality}]
Since $\psi$ is strictly decreasing and continuous, we have
\begin{eqnarray*}
\lefteqn{\P\left[\frac{|\psi'(\theta)|}{\sqrt{w_\theta(T)}}\left(\tilde\theta_T-\theta\right)\le x\right]} \\
&=&
\P\left[\tilde\theta_T\le \frac{\sqrt{w_\theta(T)}}{|\psi'(\theta)|}x+\theta\right] \\
&=&
\P\left[\psi(\tilde\theta_T)\ge \psi\left(\frac{\sqrt{w_\theta(T)}}{|\psi'(\theta)|}x+\theta\right)\right] \\
&=&
\P\left[\psi(\tilde\theta_T)-\psi(\theta)\ge \psi\left(\frac{\sqrt{w_\theta(T)}}{|\psi'(\theta)|}x+\theta\right)-\psi(\theta)\right] \\
&=&
\P\left[Q^\theta_T + \varepsilon_T^\theta\ge \psi\left(\frac{\sqrt{w_\theta(T)}}{|\psi'(\theta)|}x+\theta\right)-\psi(\theta)\right] \\
\end{eqnarray*}
Let us then introduce the short-hand notation
$$
\nu  = 
\frac{\psi\left(\frac{\sqrt{w_\theta(T)}}{|\psi'(\theta)|}x+\theta\right)-\psi(\theta)}{\sqrt{w_\theta(T)}}.
$$
By using the calculation and the short-hand notation above, we split
\begin{eqnarray*}
\lefteqn{\left|\P\left[\frac{|\psi'(\theta)|}{\sqrt{w_\theta(T)}}\left(\tilde\theta_T-\theta\right)\le x\right]-\Phi(x)\right|} \\
&=&
\bigg|\P\left[\frac{Q^\theta_T + \varepsilon_T^\theta}{\sqrt{w_\theta(T)}}\ge  \nu \right] - \Phi(x)\bigg| \\
&\le&
\Bigg|\P\left[
\frac{Q^\theta_T+\varepsilon_\theta(T)}{\sqrt{w_\theta(T)}}
\ge 
\nu \right]
-
\bar\Phi\left(\nu \right)\Bigg| 
+ \Bigg|
\bar\Phi\left(\nu \right)
-\Phi(x)
\Bigg| \\
&=& A_1 + A_2.
\end{eqnarray*}
For the term $A_1$, we split again
\begin{eqnarray*}
A_1
&=& \left|\P\left[
\frac{Q^\theta_T+\varepsilon_\theta(T)}{\sqrt{w_\theta(T)}}
\ge 
\nu 
\right]
-
\bar\Phi\left(\nu \right)\right| \\
&\le& 
\bigg|
\P\left[
\frac{Q^\theta_T}{\sqrt{w_\theta(T)}}
\ge 
\nu -\frac{\varepsilon_\theta(T)}{\sqrt{w_\theta(T)}}
\right]
-
\bar\Phi\left( \nu -\frac{\varepsilon_\theta(T)}{\sqrt{w_\theta(T)}}\right)
\bigg|
\\ 
& & +
\bigg|
\bar\Phi\left(\nu -\frac{\varepsilon_\theta(T)}{\sqrt{w_\theta(T)}}\right)
- \bar\Phi\big(\nu \big) 
\bigg| \\
&=&
A_{1,1}+A_{1,2}.
\end{eqnarray*}
By the Berry--Esseen bound of Lemma \ref{lmm:Q-conv},
$
A_{1,1} \le C R_\theta(T).
$
Consider then $A_{1,2}$. Since $|\bar\Phi(x)-\bar\Phi(y)|\le |x-y|$, we have
$$
A_{1,2} \le \frac{\varepsilon_\theta(T)}{\sqrt{w_\theta(T)}}.
$$
By the Cauchy--Schwarz inequality $|r_\theta(t)|\le r_\theta(0)=\psi(\theta)$. Consequently,
\begin{eqnarray*}
\varepsilon_\theta(T)
&=& \psi(\theta) \frac{1}{T}\int_0^T \e^{2\theta t}\, \d t
- \frac{2}{T} \int_0^T \e^{-\theta t} r_\theta(t)\, \d t \\
&\le& 
\psi(\theta) \frac{1}{T}\int_0^T \left[\e^{-2\theta t} + \e^{-\theta t}\right]\, \d t \\
&\le&
\frac{C_\theta}{T}.
\end{eqnarray*}
Therefore,
\begin{eqnarray*}
A_{1,2} 
&\le&
\frac{C_\theta/T}{\sqrt{1/T^2 \int_0^T\!\int_0^T r_\theta(t-s)^2\, \d s \d t}} \\
&=&
\frac{C_\theta}{\sqrt{\int_0^T\!\int_0^T r_\theta(t-s)^2\, \d s \d t}} \\
&\le&
C_\theta\frac{\int_0^T |r_\theta(t)|\, \d t}{\sqrt{\int_0^T\!\int_0^T r_\theta(t-s)^2\, \d s \d t}}, 
\end{eqnarray*}
where the last inequality follows from the fact that $r_\theta(0)>0$ and we can assume that $T$ is greater than some absolute constant.

Finally, it remains to consider the term $A_2$. For this, recall that $\psi$ is smooth.  Therefore, by the mean value theorem, there exists some number $\eta\in [\theta, \theta+ \frac{\sqrt{w_\theta(T)}}{|\psi'(\theta)|} x]$ such that
\begin{eqnarray*}
\nu  
&=&
\frac{1}{\sqrt{w_\theta(T)}} \, \left[ 
\psi\left(\frac{\sqrt{w_\theta(T)}}{|\psi'(\theta)|}x+\theta\right)-\psi(\theta)\right] \\
&=&
\frac{1}{\sqrt{w_\theta(T)}} \,
\psi'(\eta)\left[\frac{\sqrt{w_\theta(T)}}{|\psi'(\theta)|} x\right] 
\\
&=&
\frac{\psi'(\eta)}{|\psi'(\theta)|} x.
\end{eqnarray*}
Furthermore, since $\psi$ is decreasing, we have
$$
\frac{\psi'(\eta)}{|\psi'(\theta)|} =
-\frac{\psi'(\eta)}{\psi'(\theta)}.
$$
Consequently,
$$
\bar\Phi\big(\nu \big) =
\Phi\left(\frac{\psi'(\eta)}{\psi'(\theta)} x\right).
$$
Suppose then that $x\in[-K,K]$. Then
\begin{eqnarray*}
A_2 &=&
\left|\bar\Phi(\nu )-\Phi(x)\right| \\
&=&
\left|\Phi\left(\frac{\psi'(\eta)}{\psi'(\theta)} x\right)-\Phi(x)\right| \\
&\le&
\left|\frac{\psi'(\eta)}{\psi'(\theta)} x-x\right| \\
&\le&
\frac{K}{|\psi'(\theta)|}\left|\psi'(\eta)-\psi'(\theta)\right|.
\end{eqnarray*}
By using the mean value theorem again, we find some $\tilde\eta\in [\theta,\eta]$ such that
\begin{eqnarray*}
A_2 &\le&
\frac{K}{|\psi'(\theta)|}|\psi''(\tilde\eta)|
\frac{\sqrt{w_\theta(T)}}{|\psi'(\theta)|} \\
&\le& C_{\theta,K} \sqrt{w_\theta(T)}.
\end{eqnarray*}
Therefore, it remains to show that
$$
\sqrt{w_\theta(T)} \le 
C_\theta R_\theta(T),
$$
which translates into showing that
$$
\frac{2}{T}\int_0^T\!\!\!\!\int_0^t r_\theta(t-s)^2\, \d s \d t \le 
C_\theta\int_0^T |r_\theta(t)|\, \d t.
$$ 
Since $r_\theta(t)^2\le \psi(\theta)|r_\theta(t)|$, the inequality above follows by applying l'H\^opital's rule to it. This finishes the proof of Theorem \ref{thm:normality}
\end{proof}

We end this section with corollaries that makes Theorem \ref{thm:normality} somewhat easier to use in applications. Corollary \ref{cor:normality} deals with the classical $\sqrt{T}$ rate of convergence and Corollary \ref{cor:mixed} deals with mixed models.

\begin{cor}[Classical Rate]\label{cor:normality}
Suppose Assumption \ref{ass:invertibility} holds. Assume $\int_0^\infty r_\theta(t)^2\, \d t < \infty$. Denote
$$
\sigma^2(\theta) = 4\frac{\int_0^\infty r_\theta(t)^2\, \d t}{|\psi'(\theta)|}.
$$
Then for each $K>0$ there exists a constant $C_{\theta,K}$ such that
\begin{eqnarray*}
\lefteqn{\sup_{x\in[-K,K]}\left|
\P\left[\frac{\sqrt{T}}{\sigma(\theta)}\left(\tilde\theta_T-\theta\right)\le x\right] - \Phi(x)
\right|} \\
&\le&
C_{\theta,K}\left(\frac{1}{\sqrt{T}}\int_0^T |r_\theta(t)|\, \d t
+ \sqrt{\int_T^\infty r_\theta(t)^2\, \d t}
+ \sqrt{\frac{1}{T}\int_0^T r_\theta(t)^2 t\, \d t}\right).
\end{eqnarray*}
\end{cor}

\begin{proof}
First note that Assumption \ref{ass:ergodic} is implied by the assumption that $\int_0^\infty r_\theta(t)^2\, \d t < \infty$. Then, let us split
\begin{eqnarray*}
\lefteqn{\left|\P\left[\frac{\sqrt{T}(\tilde\theta_T-\theta)}{\sigma(\theta)}\le x\right]-\Phi(x)\right|} \\
&\le& 
\left|\P\left[\frac{|\psi'(\theta)|(\tilde\theta_T-\theta)}{\sqrt{w_\theta(T)}}\le \frac{x 2\sqrt{\int_0^\infty r_\theta(t)^2\, \d t}}{\sqrt{T w_\theta(T)}}\right] -\Phi\left(\frac{x 2\sqrt{\int_0^\infty r_\theta(t)^2\, \d t}}{\sqrt{T w_\theta(T)}}\right)\right| \\
& & +
\left|\Phi\left(\frac{x2\sqrt{\int_0^\infty r_\theta(t)^2\, \d t}}{\sqrt{T w_\theta(T)}}\right)-\Phi(x)\right| \\
&=&
A_1 + A_2.
\end{eqnarray*}
Now
$$
Tw_T(\theta) \sim 4\int_0^T r_\theta(t)^2\, \d t \sim 4 \int_0^\infty r_\theta(t)^2\, \d t,
$$
i.e., $Tw_T(\theta)$ is asymptotically a positive constant. Consequently, we can take the supremum over $x$ on a compact interval, and Theorem \ref{thm:normality} implies that the term $A_1$ is dominated by
$$
C_{\theta,K}R_\theta(T) 
\le
C_{\theta,K} \frac{\int_0^T | r_\theta(t)|\, \d t}{\sqrt{T\int_0^T r_\theta(t)^2\, \d t}} \\
\le
C_{\theta,K} \frac{1}{\sqrt{T}}\int_0^T | r_\theta(t)|\, \d t. 
$$
For the second term, we use the estimate
$$
\sup_{x\in\R} |\Phi(\rho x) - \Phi(x)| \le 
|\rho-1|. 
$$
Setting
$$
\rho = 
\frac{2\sqrt{\int_0^\infty r_\theta(t)^2\, \d t}}{\sqrt{T w_\theta(T)}} 
=
\frac{\sqrt{\int_0^\infty r_\theta(t)^2\, \d t}}{\sqrt{\frac{1}{T}\int_0^T\!\!\int_0^t r_\theta(s)^2\, \d s \d t}},
$$
we obtain for the term $A_2$ the upper bound
\begin{eqnarray*}
\lefteqn{\left|
\frac{\sqrt{\int_0^\infty r_\theta(t)^2\, \d t}-\sqrt{\frac{1}{T}\int_0^T\!\!\int_0^t r_\theta(s)^2\, \d s \d t}}{\sqrt{\frac{1}{T}\int_0^T\!\!\int_0^t r_\theta(s)^2\, \d s \d t}}
\right|} \\
&\le&
C_{\theta} \sqrt{\left|\int_0^\infty r_\theta(t)^2\, \d t - \frac{1}{T}\int_0^T\!\!\!\!\int_0^t r_\theta (s)^2 \,\d s \d t\right|} \\
%&=&
%C_{\theta} \sqrt{\left|\int_0^\infty r_\theta(t)^2\, \d t - \frac{1}{T}\int_0^T\!\!\!\!\int_t^T r_\theta (t)^2 \d s \d t\right|} \\
&=&
C_{\theta} \sqrt{\left|\int_0^\infty r_\theta(t)^2\, \d t - \frac{1}{T}\int_0^T r_\theta (t)^2 (T-t)\,\d t\right|} \\
%&=&
%C_{\theta} \sqrt{\left|\int_0^\infty r_\theta(t)^2\, \d t - \int_0^T r_\theta (t)^2 \,\d t + \frac{1}{T}\int_0^T r_\theta (t)^2 t\, \d t\right|} \\
&=&
C_{\theta} \sqrt{\left|\int_T^\infty r_\theta(t)^2\, \d t + \frac{1}{T}\int_0^T r_\theta (t)^2 t\, \d t\right|} \\
&\le&
C_{\theta}\left(\sqrt{\int_T^\infty r_\theta(t)^2\, \d t} + \sqrt{\frac{1}{T}\int_0^T r_\theta (t)^2 t\, \d t}\right),
\end{eqnarray*}
since $|\sqrt{a}-\sqrt{b}|\le \sqrt{|a-b|}$ and $\sqrt{a+b}\le\sqrt{a}+\sqrt{b}$. 
\end{proof}

\begin{cor}[Mixed Models]\label{cor:mixed}
Let $G^i$, $i=1,\ldots,n$, be independent continuous stationary-increment Gaussian processes with zero mean each satisfying Assumption \ref{ass:ergodic} and Assumption \ref{ass:invertibility}. Let $r_{\theta,i}$ be the autocovariance of the stationary solution corresponding the noise $G^i$. Assume that $r_{\theta,i}\ge 0$ for all $i$. Then, for the noise $G=\sum_{i=1}^n G^i$, there exists, for any $K>0$, a constant $C_{\theta,K}$ such that
\begin{eqnarray*}
\lefteqn{\sup_{x\in[-K,K]}\left|
\P\left[\frac{|\psi'(\theta)|}{\sqrt{w_\theta(T)}}\left(\tilde\theta_T-\theta\right)\le x\right] - \Phi(x)
\right| } \\
&\le&  
C_{\theta,K}\max_{i=1,\ldots,n} 
\frac{\int_0^T r_{\theta,i}(t)\, \d t}{\sqrt{\int_0^T\!\!\int_0^T r_{\theta,i}(t-s)^2\, \d s \d t}}.
\end{eqnarray*}
\end{cor}

\begin{proof}
Since the $G^i$'s are independent, the autocovariance for the mixed model with noise $G$ is $r_\theta = \sum_{i=1}^n r_{\theta,i}$.  Consequently, Assumption \ref{ass:ergodic} and Assumption \ref{ass:invertibility} hold.
It remains to show that
$$
 \frac{\int_0^T \sum_{i=1}^n r_{\theta,i}(t)\, \d t}{\sqrt{\int_0^T\!\!\int_0^T \left(\sum_{i=1}^n r_{\theta,i}(t-s)\right)^2\, \d s \d t}}
\le \max_{i=1,\ldots,n} 
\frac{n\int_0^T r_{\theta,i}(t)\, \d t}{\sqrt{\int_0^T\!\!\int_0^T r_{\theta,i}(t-s)^2\, \d s \d t}}
$$
The case for the nominator is clear.  For the denominator, we use the fact that the $r_{\theta,i}$'s are non-negative. Indeed, then
$$
\left(\sum_{i=1}^n r_{\theta,i}(t-s)\right)^2 \ge r_{\theta,i}(t-s)^2
$$
for any $i$, as the cross-terms $r_{\theta,i}(t-s)r_{\theta,j}(t-s)$ are positive.
\end{proof}

\begin{rem}
It seems challenging to obtain a Berry--Esseen type bound for $\sup_{x\in\R}$ instead of $\sup_{x\in [-K,K]}$ in Theorem \ref{thm:normality} and its corollaries. Indeed, the same problem was present in Hu and Song \cite{Hu-Song-2013}, where the Berry--Esseen bound was provided in the case of fractional Ornstein--Uhlenbeck process of the first kind. The reason for this is that by using the second order Taylor expansion we obtain that the essential factor we are estimating is 
$$
\P\left[\frac{Q^\theta_T}{\sqrt{w_\theta(T)}} \le x + C\frac{\psi''(\eta)}{w_\theta(T)}x^2\right] - \Phi(x), 
$$
where $\eta$ is a number between $\theta$ and $\theta +x/w_\theta(T)$.
 Now, since $\psi''$ is not bounded, and 
 $$
 \P\left[\frac{Q^\theta_T}{\sqrt{w_\theta(T)}} \le x\right] \approx \Phi(x),
$$
the probability 
$$
\P\left[\frac{Q^\theta_T}{\sqrt{w_\theta(T)}} \le x + C\frac{\psi''(\eta)}{w_\theta(T)}x^2\right]
$$
increases much faster than $\Phi(x)$ due to the term involving $x^2$.
\end{rem}

%%%%%%%%%%%%%%%%%%%%%%%%%%%%%%%%%%%%%%%%%%%%%%%%%%%%%%%%%%%%%%%%%%%%%%%%%%%%%%%
\subsection{Discrete Observations}

In practice continuous observations are rarely available.  Therefore, it is important to consider the case of discrete observations. To control the error introduced by the unobserved time-points, we assume that the driving noise $G$ is H\"older continuous with some index $H\in (0,1)$, i.e., $G$ is H\"older continuous with parameter $\gamma$ for all $\gamma<H$.  The general idea is, that the smaller the $H$ the more care must be taken in choosing the time-mesh of the observations. This gives rise to the condition \eqref{eq:mesh} in Theorem \ref{thm:discrete}. 

Note that from the form of the Langevin equation it is immediate that any solution is H\"older continuous with index $H$ if and only if the driving noise is H\"older continuous with the same index $H$.  Due to  \cite[Corollary 1]{Azmoodeh-Sottinen-Viitasaari-Yazigi-2014}, the following assumption is not only sufficient, but also necessary, for the H\"older continuity with index $H$:

\begin{ass}[H\"older continuity]\label{ass:holder}
Let $H\in(0,1)$. For all $\varepsilon>0$ there exists a constant $C_\varepsilon$ such that
$$
v(t) \le C_\varepsilon \, t^{2H-\varepsilon}.
$$ 
\end{ass}

For notational simplicity, we assume equidistant observation times $t_k= k\Delta_N$, $k=0,\ldots, N$. Denote $T_N=N\Delta_N$ and assume that $\Delta_N \to 0$ with $T_N\to \infty$.
The AE based on the discrete observations is
$$%\begin{equation}\label{eq:AE-discrete}
\tilde\theta_N = \psi^{-1}\left(\frac{1}{T_N}\sum_{k=1}^N (X^\theta_{k\Delta_N})^2\Delta_N\right).
$$%\end{equation}

\begin{thm}[Discrete Observations]\label{thm:discrete}
Suppose Assumption \ref{ass:ergodic}, Assumption \ref{ass:invertibility} and Assumption \ref{ass:holder} hold.  Assume further that
\begin{equation}\label{eq:mesh}
N \Delta_N^{\beta} \to 0,
\end{equation}
where 
$$
\beta = \beta(H) = \frac{2H+\frac12}{H+\frac12} - \delta
$$
for some $\delta>0$.
Then,
$$
\tilde\theta_N \to \theta \quad{a.s.}
$$
Moreover, if Assumption \ref{ass:normal} holds, then
$$
\frac{|\psi'(\theta)|}{\sqrt{w_\theta(T_N)}}\left(\tilde\theta_N-\theta\right)
\stackrel{\d}{\to} \mathcal{N}\left(0,1\right).
$$ 
\end{thm}

\begin{proof}
Following the proof of \cite[Theorem 3.2]{Azmoodeh-Viitasaari-2015a}, 
it is enough to show that
\begin{equation}\label{eq:discrete}
\frac{1}{\sqrt{w_\theta(T_N)}}\left(
\frac{1}{T_N}\sum_{k=1}^N (X^\theta_{k\Delta_N})^2\Delta_N - 
\frac{1}{T_N}\int_0^{T_N} (X^\theta_t)^2\, \d t
\right) \to 0 \quad\mbox{a.s.}
\end{equation}

Let
$$
Y_k = \sup_{t,s\in [t_{k-1},t_k]} \frac{|X^\theta_t-X^\theta_s|}{|t-s|^{H-\varepsilon}}
$$
be the $(H-\varepsilon)$-H\"older constant of the process $X^\theta$ on the subinterval $[t_{k-1},t_k]$. Similarly, let (with slight abuse of notation) $Y_{N}$ be the $(H-\varepsilon)$-H\"older constant of the process $X^\theta$ on the entire interval $[0,T_N]$
Then, by the identity $a^2-b^2=(a+b)(a-b)$, the H\"older continuity of $X^\theta$, and the triangle inequality,
\begin{eqnarray*}
\lefteqn{\frac{1}{\sqrt{w_\theta(T_N)}}\left|\frac{1}{T_N}\sum_{k=1}^N (X^\theta_{t_k})^2\Delta_N - 
\frac{1}{T_N}\int_0^{T_N} (X^\theta_t)^2\, \d t\right|} \\
&\le& 
\frac{1}{T_N\sqrt{w_\theta(T_N)}}
\sum_{k=1}^N \int_{t_{k-1}}^{t_{k}} \big|(X^\theta_{t_k})^2 -  (X^\theta_t)^2\big|\, \d t \\
&\le&
\frac{2}{T_N\sqrt{w_\theta(T_N)}}
\sum_{k=1}^N \sup_{t\in [t_{k-1},t_k]} \big|X^\theta_t\big|\,
\int_{t_{k-1}}^{t_{k}}  \big|X^\theta_{t_k} -  X^\theta_t\big|\, \d t \\
&\le&
\frac{2}{T_N\sqrt{w_\theta(T_N)}}
\sum_{k=1}^N \sup_{t\in [t_{k-1},t_k]} \big|X^\theta_t\big|\,
Y_k\int_{t_{k-1}}^{t_{k}}  \big|t -t_{k-1}\big|^{H-\varepsilon}\, \d t \\
&=&
\frac{C \Delta_N^{H-\varepsilon+1}}{T_N\sqrt{w_\theta(T_N)}}
\sum_{k=1}^N \sup_{t\in [t_{k-1},t_k]} \big|X^\theta_t\big|\,
Y_k. 
%\\
%&\le&
%\frac{C \Delta_N^{H-\varepsilon+1}T_N^{H-\varepsilon}Y_N}{T_N\sqrt{w_\theta(T_N)}}
%\sum_{k=1}^N  
%Y_k.
\end{eqnarray*}
Note that
\begin{eqnarray*}
\sum_{k=1}^N 
\sup_{t\in[t_{k-1},t_k]} |X^\theta_t| 
&\le& T_N^{H-\varepsilon} Y_N, \\
\sum_{k=1}^N Y_k^2 &\le& N Y_{N}^2, \\
w_\theta(T) &\ge& C T^{-1},
\end{eqnarray*}
where the last estimate follows, e.g., from Lemma \ref{lmm:equiv-var}. Consequently, it remains to show that
\begin{equation}\label{eq:discrete-limit}
N^{H-\varepsilon+\frac12}
\Delta_N^{2H-2\varepsilon+\frac12}
\, Y_N^2
\to 0 \quad\mbox{a.s.}
\end{equation}
By \cite[Theorem 1 and Lemma 2]{Azmoodeh-Sottinen-Viitasaari-Yazigi-2014} (see also \cite[Remark 2.3]{Azmoodeh-Viitasaari-2015a}) we have for all $p\ge 1$ a constant $C=C_{\theta,H,\varepsilon,p}$ such that
$$
\E\left[Y_{N}^{2p}\right]
\le 
C T_N^{2 \varepsilon p}.
$$
From this estimate and from the Markov's inequality it follows that for all $y>0$ and $p\ge 1$, 
$$
\P\left[\frac{Y_{N}^2}{N^\varepsilon} > y\right]
\le
\frac{C}{y^p} \left(\Delta_N^{2\varepsilon}N^{\varepsilon}\right)^p.
$$
Now, by choosing  $p$ large enough, we obtain
$$
\sum_{N=1}^\infty \P\left[\frac{Y_{N}^2}{N^\varepsilon} > y\right]
<
\infty 
$$
if
$$
\Delta_N^{2\varepsilon} N^{\varepsilon} \le N^{-\alpha}
$$
for some $\alpha>0$. By \eqref{eq:mesh}, we may choose
$
\alpha = 2\varepsilon/\beta-\varepsilon.
$
Indeed, since $\beta<2$, it follows that $\alpha>0$. Consequently, by the Borel--Cantelli lemma $N^{-\varepsilon}Y_N^2 \to 0$ almost surely. By applying this to \eqref{eq:discrete-limit} it remains to show that
$$
N^{H+\frac12}
\Delta_N^{2H-2\varepsilon+\frac12}
\to 0.
$$
But this follows from \eqref{eq:mesh} by choosing $\varepsilon<\min\{ H + 1/4, \delta(H+1/2)/2\}$. The details are left to the reader.
\end{proof}

\begin{rem}
The Berry--Esseen bound for Theorem \ref{thm:discrete} can be obtained as in the proof above by analyzing the speed of convergence in \eqref{eq:discrete}.  We leave the details for the reader.
\end{rem}

%%%%%%%%%%%%%%%%%%%%%%%%%%%%%%%%%%%%%%%%%%%%%%%%%%%%%%%%%%%%%%%%%%%%%%%%%%%%%%%
%%%%%%%%%%%%%%%%%%%%%%%%%%%%%%%%%%%%%%%%%%%%%%%%%%%%%%%%%%%%%%%%%%%%%%%%%%%%%%%
\section{Examples}\label{sect:examples}

\subsection{Fractional Ornstein--Uhlenbeck Process of the First Kind}

The fractional Brownian motion $B^H$ with Hurst index $H\in(0,1)$ is the stationary-increment Gaussian process with variance function $v_H(t) = t^{2H}$.  Actually, it is the (upto a multiplicative constant) unique stationary-increment Gaussian process that is $H$-self-similar meaning that
$$
B^H \stackrel{\d}{=} a^{-H} B^H_{a\,\cdot} 
$$
for all $a>0$.
For the fractional Brownian motion the Hurst index $H$ is both the index of self-similarity and the H\"older index.  We refer to Biagini et al. \cite{Biagini-Hu-Oksendal-Zhang-2008} and Mishura \cite{Mishura-2008} for more information on the fractional Brownian motion.
The fractional Ornstein--Uhlenbeck process (of the first kind) is the stationary solution to the Langevin equation 
\begin{equation}\label{eq:fou}
\d U^{H,\theta}_t = -\theta U^{H,\theta}_t\, \d t + \d B^H_t, \quad t\ge 0.
\end{equation} 
The fractional Ornstein--Uhlenbeck processes (of different kinds) and related parameter estimations have been studied extensively recently, see, e.g., \cite{Azmoodeh-Viitasaari-2015a,Cheridito-Kawaguchi-Maejima-2003,Hu-Nualart-2010,Hu-Song-2013,Kaarakka-Salminen-2011,Kleptsyna-LeBreton-2002,Sottinen-Tudor-2008}. 
By Cheridito et al. \cite[Theorem 2.3]{Cheridito-Kawaguchi-Maejima-2003} the autocovariance $r_{H,\theta}$ of the stationary solution satisfies, for $H\ne1/2$, the asymptotic expansion
\begin{equation}\label{eq:asy-r-fou}
r_{H,\theta}(t) \sim \frac{H(2H-1)}{\theta^2} t^{2H-2}
\end{equation}
as $t\to\infty$.  Also, by Hu and Nualart \cite{Hu-Nualart-2010}, 
$$
\psi_{H}(\theta) = \frac{H\Gamma(2H)}{\theta^{2H}},
$$
where $\Gamma$ is the Gamma function. Consequently, Assumption \ref{ass:ergodic} and Assumption \ref{ass:invertibility} are satisfied for all $H$, and Assumption \ref{ass:normal} is satisfied for $H\le 3/4$. Also, Assumption \ref{ass:holder}, required for discrete observations, is satisfied for all $H$. Finally, we observe that Corollary \ref{cor:normality} is applicable for $H\in(0,3/4)$, and by using the self-similarity of the fractional Brownian motion it is clear that
$$
\int_0^\infty r_{H,\theta}(t)^2\, \d t = \theta^{-2H}\sigma_H^2,
$$
where we have denoted
\begin{equation}\label{eq:sigma-H}
\sigma_H^2 = \int_0^\infty r_{H,1}(t)^2\, \d t.
\end{equation}
Let $\tilde\theta^H_{T}$ be the AE associated with the equation \eqref{eq:fou}.  Proposition \ref{pro:fou} below extends the result of Hu and Nualart \cite[Theorem 4.1]{Hu-Nualart-2010} both by extending the range of $H$ and by providing the Berry--Esseen bounds.

\begin{pro}[Fractional Ornstein--Uhlenbeck Process of the First Kind]\label{pro:fou}
Let $K>0$ arbitrary. Let $\sigma_H$ be given by \eqref{eq:sigma-H}.
\begin{enumerate}
\item
Let $H\in (0,1/2]$. Then
$$
\sup_{x\in[-K,K]}\left|\P\left[\sqrt{\frac{T}{\theta\sigma_H^2}}\left(\tilde\theta^H_T-\theta\right)\le x\right] - \Phi(x)\right| 
\le
\frac{C_{H,\theta,K}}{\sqrt{T}}.
$$
\item
Let $H\in(1/2,3/4)$. Then 
$$
\sup_{x\in[-K,K]}\left|\P\left[\sqrt{\frac{T}{\theta\sigma_H^2}}\left(\tilde\theta^H_T-\theta\right)\le x\right] - \Phi(x)\right| 
\le
\frac{C_{H,\theta,K}}{\sqrt{T^{3-4H}}}.
$$
\item
Let $H=3/4$. Then 
$$
\sup_{x\in[-K,K]}\left|\P\left[\sqrt{\frac{T}{\theta \sigma^2\log T}}\left(\tilde\theta^{3/4}_T-\theta\right)\le x\right] - \Phi(x)\right| 
\le
\frac{C_{3/4,\theta,K}}{\sqrt{\log T}},
$$
where $\sigma$ is an absolute constant.  
\end{enumerate}
\end{pro} 

\begin{proof}
Consider first the case $H\in (0, 1/2)$. By Corollary \ref{cor:normality}, it is enough to show that 
$$
\frac{1}{\sqrt{T}}\int_0^T |r_{H,\theta}(t)|\, \d t
+ \sqrt{\int_T^\infty r_{H,\theta}(t)^2\, \d t}
+ \sqrt{\frac{1}{T}\int_0^T r_{H,\theta}(t)^2 t\, \d t}
\le
\frac{C_{H,\theta}}{\sqrt{T}}.
$$
Here the first term is the dominating one.  Indeed, by \eqref{eq:asy-r-fou},
\begin{eqnarray*}
\frac{1}{\sqrt{T}}\int_0^\infty |r_{H,\theta}(t)|\, \d t 
&\sim&
\frac{1}{\sqrt{T}}C_{H,\theta} \int_1^\infty t^{2H-2}\, \d t \\
&\le& \frac{C_{\theta,H}}{\sqrt{T}},
\end{eqnarray*}
\begin{eqnarray*}
\sqrt{\int_T^\infty r_{H,\theta}(t)^2 \, \d t} 
&\sim& C_{H,\theta} \sqrt{\int_T^\infty t^{4H-4}\, \d t} \\ 
&=& C_{H,\theta} \sqrt{T^{4H-3}},
%  \\
%&\le& \frac{C_{H,\theta}}{\sqrt{T}},
\end{eqnarray*}
\begin{eqnarray*}
\sqrt{\frac{1}{T}\int_0^T r_{H,\theta}(t)^2 t\, \d t} 
&\sim& C_{H,\theta} \sqrt{\frac{1}{T} \int_1^T t^{4H-3}\, \d t} \\
&=& C_{H,\theta} \sqrt{T^{4H-3}}
%\\
%&\le& \frac{C_{H,\theta}}{\sqrt{T}}. 
\end{eqnarray*}

The case $H=1/2$ is classical and well-known, and stated here only for the sake of completeness.

The case $H\in(1/2,3/4)$ can be analyzed exactly the same way as the case $H\in(0,1/2)$, except now it is the second and third terms that dominate.

Consider then the case $H=3/4$. Now Corollary \ref{cor:normality} is not applicable.  Consequently, we have to use Theorem \ref{thm:normality} directly.  Let us first calculate the asymptotic rate. By applying l'H\^opital's rule twice and then the asymptotic expansion \eqref{eq:asy-r-fou}, we obtain
\begin{eqnarray*}
w_{3/4,\theta}(T)
&=&
\frac{2}{T^2} \int_0^T\!\!\!\!\int_0^T r_{3/4,\theta}(t-s)^2\, \d s \d t \\
&\sim& \frac{2\log T}{T} \frac{1}{\log T}\int_0^T r_{3/4,\theta}(t)^2 \, \d t \\
&\sim& 
\frac{2(3/8)^2}{\theta^4}\frac{\log T}{T} \frac{1/T}{1/T} \\
&=&\frac{2(3/8)^2}{\theta^4}\frac{\log T}{T}, 
\end{eqnarray*}
Consequently,
$$
\frac{|\psi_{3/4}'(\theta)|}{\sqrt{w_{3/4,\theta}(T)}}
\sim 
\frac{\frac34\Gamma(\frac32)\frac32 \, \theta^{-5/2}}{\sqrt{2}\frac{3}{8} \theta^{-2} } \sqrt{\frac{T}{\log T}} 
= \sqrt{\frac{T}{\theta\sigma^2\log T}},
$$ 
by setting $\sigma^2$ appropriately. For the Berry--Esseen upper bound, we estimate
$$
\frac{\int_0^T |r_{3/4,\theta}(t)|\, \d t}{T\sqrt{w_{3/4,\theta}(T)}}
\le 
C_\theta
\frac{\int_0^T t^{-1/2}\, \d t}{\sqrt{T}\sqrt{\log T}}
\le
C_\theta
\frac{\sqrt{T}}{\sqrt{T}\sqrt{\log T}}
$$
The claim follows.
\end{proof}

\begin{rem}
In the case $H>3/4$ our method does not provide asymptotic normality.  Indeed, due to the results in Breton and Nourdin \cite{Breton-Nourdin-2008} it is expected that asymptotic normality cannot hold in this case.
\end{rem}

%%%%%%%%%%%%%%%%%%%%%%%%%%%%%%%%%%%%%%%%%%%%%%%%%%%%%%%%%%%%%%%%%%%%%%%%%%%%%%%
\subsection{Noises Arising from Self-Similar Processes}

The examples in the next two subsections deal with self-similar processes.  The motivation comes from the result of Doob \cite{Doob-1942} stating that the classical Ornstein--Uhlenbeck process can be viewed as the inverse Lamperti transform of the Brownian motion that is $1/2$-self-similar. 
Therefore, let us start with an $H$-self-similar Gaussian process $Y^{H}$. (For a representation of such processes in terms of the Brownian motion see Yazigi \cite{Yazigi-2015}). The inverse Lamperti transform with self-similarity parameters $H$ and scale parameter $\theta$ is
$$
(\L_{H,\theta}^{-1} Y^H)_t = \e^{-\theta t} Y^H_{a(t)},  
$$
where
$$
a(t) = a_{H,\theta}(t) = \frac{H}{\theta} \e^{\frac{\theta}{H}t}.
$$
If $Y^H$ is $H$-self-similar, then $\L_{H,\theta}^{-1}Y^H$ is stationary, and vice versa. Furthermore, and all stationary solutions $U^\theta=U^{H,\theta}$ of the Langevin equation are inverse Lamperti transforms of some $H$-self-similar $Y^H$, see \cite{Lamperti-1962,Viitasaari-2014-preprint}. Therefore we have, on the one hand,
$$
U^{H,\theta}_t = \e^{-\theta t} Y^H_{a(t)}
$$
for some $H$-self-similar $Y^H$, and, on the other hand,
$$
U^{H,\theta}_t = \int_{-\infty}^t \e^{-\theta(t-s)}\, \d G^{H,\theta}_s
$$
for some stationary-increment noise $G^{H,\theta}$ arising from the $H$-self-similar process $Y^H$. Actually, we have
$$
G^{H,\theta}_t = \int_0^t \e^{-\theta s}\, \d Y^H_{a(t)}.
$$
Moreover, the family $\{ G^{H,\theta}\,;\, \theta>0\}$ satisfies the scaling property
$$
\theta^H G^{H,\theta}_{\cdot\,/\theta} \stackrel{\d}{=} G^H, 
$$
where we have denoted $G^H=G^{H,1}$.  This motivates the study of the following Langevin equation where the noise arises from the $H$-self-similar process $Y^H$ that depends on $\theta$, but with a noise $G^H$ that is independent of $\theta$:
\begin{equation}\label{eq:second-kind}
\d U^{H,\theta}_t = -\theta U^{H,\theta}_t \d t + \d G^{H}_t.
\end{equation}
The solutions of these Langevin equations are called Ornstein--Uhlenbeck processes of the second kind.

%%%%%%%%%%%%%%%%%%%%%%%%%%%%%%%%%%%%%%%%%%%%%%%%%%%%%%%%%%%%%%%%%%%%%%%%%%%%%%%
\subsection{Fractional Ornstein--Uhlenbeck Process of the Second Kind}

The fractional Ornstein--Uhlenbeck process of the second kind arises from \eqref{eq:second-kind} by setting $G^H=B^H$, the fractional Brownian motion. This process has been studied e.g. in \cite{Azmoodeh-Morlanes-2015,Azmoodeh-Viitasaari-2015a,Kaarakka-Salminen-2011}. 
By Kaarakka and Salminen \cite[Proposition 3.11]{Kaarakka-Salminen-2011} the autocovariance $r_{H,\theta}$ of the fractional Ornstein--Uhlenbeck process of the second kind has exponential decay. Consequently, Corollary \ref{cor:normality} implies the following:

\begin{pro}
For the fractional Ornstein--Uhlenbeck processes of the second find
$$
\sup_{x\in[-K,K]}\left|\P\left[\frac{\sqrt{T}}{\sigma_{H,\theta}(\theta)}\left(\tilde{\theta}^{H}_T-\theta\right) \le x\right]-\Phi(x)\right|
\le \frac{C_{H,\theta,K}}{\sqrt{T}},
$$
where
$$
\sigma_{H,\theta}(\theta)
=
\sqrt{4\int_0^\infty r_{H,\theta}(t)^2\, \d t}.
$$
\end{pro} 

%%%%%%%%%%%%%%%%%%%%%%%%%%%%%%%%%%%%%%%%%%%%%%%%%%%%%%%%%%%%%%%%%%%%%%%%%%%%%%%
\subsection{Bifractional Ornstein--Uhlenbeck Process of the Second Kind}

The bifractional Brownian motion $B^{H,K}$, introduced by Houdr\'e and Villa \cite{Houdre-Villa-2003}, with parameters $H\in(0,1)$ and $K\in(0,1]$ is the Gaussian process with covariance
$$
\E\!\left[B^{H,K}_t B^{H,K}_t\right] = \frac{1}{2^K}\left[(t^{2H}+s^{2H})^K - |t-s|^{2HK}\right].
$$
The bifractional Brownian motion is a generalization of the fractional Brownian motion, but it does not have stationary increments, except in the fractional case $K=1$. Consequently, there does not seem to be a natural way to define the bifractional Ornstein--Uhlenbeck process of the first kind that would have a stationary version.  The bifractional Brownian motion is, however, $HK$-self-similar, see Russo and Tudor \cite{Russo-Tudor-2006}. Consequently, we can define the bifractional Ornstein--Uhlenbeck process by setting $Y^{HK} = B^{H,K}$ in equation \eqref{eq:second-kind}.

\begin{lmm}\label{lmm:bifractional}
The autocovariance $r_{H,K,\theta}$ of the bifractional Ornstein--Uhlenbeck process of the second kind has exponential decay.
\end{lmm}

Lemma \ref{lmm:bifractional} combined with Corollary \ref{cor:normality} immediately yields:

\begin{pro}
For the bifractional Ornstein--Uhlenbeck processes of the second find
$$
\sup_{x\in[-L,L]}\left|\P\left[\frac{\sqrt{T}}{\sigma_{{H,K},\theta}(\theta)}\left(\tilde{\theta}^{{H,K},\theta}_T-\theta\right) \le x\right]-\Phi(x)\right|
\le \frac{C_{H,K,\theta,L}}{\sqrt{T}},
$$
where
$$
\sigma_{H,K,\theta}(\theta)
=
\sqrt{4\int_0^\infty r_{H,K,\theta}(t)^2\, \d t}.
$$
\end{pro} 

%%%%%%%%%%%%%%%%%%%%%%%%%%%%%%%%%%%%%%%%%%%%%%%%%%%%%%%%%%%%%%%%%%%%%%%%%%%%%%%
%%%%%%%%%%%%%%%%%%%%%%%%%%%%%%%%%%%%%%%%%%%%%%%%%%%%%%%%%%%%%%%%%%%%%%%%%%%%%%%
\section{Discussion on Other Estimators}\label{sect:discussion}

It is a celebrated result by Gauss \cite{Gauss-1809} that for multivariate Gaussian distributions the Least Squares Estimator (LSE) and the Maximum Likelihood Estimator (MLE) coincide, and this is a characterizing property of the Gaussian distribution.  Indeed, this is why Gaussian distributions are named thus.  In the infinite-dimensional case of Gaussian processes, the situation is more delicate, as the following discussion shows. The discussion is based on  Hu and Nualart \cite{Hu-Nualart-2010} in the case of the LSE and on Kleptsyna and Le Breton \cite{Kleptsyna-LeBreton-2002} in the case of the MLE.
To make the discussion short, we do not present explicit assumptions in terms of the variance $v$, although this would be possible. Instead, we confine ourselves in presenting the general ideas and implicit assumptions.  

%%%%%%%%%%%%%%%%%%%%%%%%%%%%%%%%%%%%%%%%%%%%%%%%%%%%%%%%%%%%%%%%%%%%%%%%%%%%%%%
\subsection{Least Squares Estimator}

In this subsection, $\delta$ denotes the Skorohod integral. We refer to Nualart \cite{Nualart-2006} for details on Skorohod integrals.

One the one hand, the LSE
\begin{equation}\label{eq:lse}
\hat\theta_T = - \frac{\int_0^T X^\theta_t\, \delta X^\theta_t}{\int_0^T (X^\theta_t)^2\, \d t}
\end{equation}
arises heuristically by minimizing 
$$
\int_0^T |\dot X^\theta_t+\theta X^\theta_t |^2\, \d t.
$$
On the other hand, by using the Langevin equation of the solution $X^\theta$, one would hope that
\begin{equation}\label{eq:skoro-eq}
\int_0^T X^\theta_t \, \delta X^\theta_t
= 
-\theta\int_0^T (X^\theta_t)^2\, \d t + \int_0^T X^\theta_t \, \delta G_t. 
\end{equation}
This would lead to the LSE
\begin{equation}\label{eq:skoro-lse}
\widehat\theta_T = \theta - \frac{\int_0^T X^\theta_t\, \delta G_t}{\int_0^T (X^\theta_t)^2\, \d t}.
\end{equation}
Unfortunately, the Skorohod integral is not (bi)linear.  In particular, the equation \eqref{eq:skoro-eq} does not hold. Indeed, this is obvious from the fact that Skorohod integrals have zero mean. Consequently, the LSE's defined by \eqref{eq:lse} and \eqref{eq:skoro-lse}, respectively, are not the same. The LSE defined by \eqref{eq:skoro-lse} has been shown to be consistent for some fractional Ornstein--Uhlenbeck processes, see \cite{Azmoodeh-Morlanes-2015,Hu-Nualart-2010}. However, the LSE \eqref{eq:skoro-lse} depends on $\theta$, the parameter we want to estimate!  Therefore, the LSE \eqref{eq:skoro-lse} does not seem to be particularly convenient. Moreover, to show that the LSE \eqref{eq:skoro-lse} is consistent, one has to show that the term $\int_0^T X^\theta_t\, \delta G_t/\int_0^T (X^\theta_t)^2\, \d t$ converges to zero.  This suggest that it would be more natural to define the LSE by \eqref{eq:lse}.  However, Proposition \ref{pro:lse}  below shows that the LSE \eqref{eq:lse} will fail under rather general assumptions.

\begin{pro}\label{pro:lse}
Assume that $U^\theta$ is ergodic, and that the Skorohod integral $\int_0^T X^\theta_t\, \delta X^\theta_t$ exists. Let $\hat\theta_T$ be defined by \eqref{eq:lse}. 
If $(X^\theta_T)^2/T\to 0$ in $L^1(\Omega)$ and almost surely, then
$$
\hat\theta_T \to 0 \quad\mbox{a.s.}
$$
\end{pro}

\begin{proof}
By the It\^o formula in \cite{Sottinen-Viitasaari-2014-preprint},
$$
\int_0^T X^\theta_t\, \delta X^\theta_t 
=\frac12 (X^\theta_T)^2 - \frac12 \E[(X^\theta_T)^2].
$$
Since $U^\theta$ is ergodic, Proposition \ref{pro:U-vs-X} implies that
$$
\frac{1}{T} \int_0^T (X^\theta_t)^2\, \d t \to \psi(\theta)>0 \quad\mbox{a.s.}
$$
Consequently,
$$
\hat\theta_T = \frac12\frac{(X^\theta_T)^2/T - \E[(X^\theta_T)^2/T]}{\frac{1}{T} \int_0^T (X^\theta_t)^2\, \d t} \to 0 \quad\mbox{a.s},
$$
and the claim follows.
\end{proof}

%%%%%%%%%%%%%%%%%%%%%%%%%%%%%%%%%%%%%%%%%%%%%%%%%%%%%%%%%%%%%%%%%%%%%%%%%%%%%%%
\subsection{Maximum Likelihood Estimator}

In this subsection, the integrals are abstract Wiener integral as  defined in e.g. \cite{Sottinen-Yazigi-2014}, or, equivalently, Skorohod integrals as defined e.g. in \cite{Sottinen-Viitasaari-2014-preprint}.

To construct the MLE, we assume the following Volterra representation for the noise $G$:
There exists a Gaussian martingale $M$ with bracket $\langle M \rangle$ and a kernel $k\in L_{\mathrm{loc}}^2(\R_+^2, \d\langle M \rangle\times\d\langle M \rangle)$ such that
\begin{equation}\label{eq:volterra}
G_t = \int_0^t k(t,s)\, \d M_s.
\end{equation}
Furthermore, we assume the following inverse Volterra representation:
\begin{equation}\label{eq:inverse-volterra}
M_t = \int_0^t k^*(t,s)\, \d G_t.
\end{equation}
Next, we define
\begin{eqnarray*}
M^\theta_t &=& \int_0^t k^*(t,s)\, \d X^\theta_t, \\
\Xi^\theta_t &=& \frac{\d}{\d\langle M \rangle_t}\int_0^t k^*(t,s) X^\theta_s\, \d s,
\end{eqnarray*}
implicitly assuming their existence.

\begin{pro}[MLE]\label{pro:mle}
Assume representations \eqref{eq:volterra}--\eqref{eq:inverse-volterra} and assume that $\Xi^\theta\in L^2(\Omega\times[0,T], \d\P\times\d\langle M \rangle)$. 
Then the MLE based on the observations $X^\theta_t$, $t\in[0,T]$, is
$$
\bar\theta_T = - \frac{\int_0^T \Xi^\theta_t\, \d M^\theta_t}{\int_0^T(\Xi^\theta_t)^2\, \d\langle M\rangle_t}.
$$
Moreover, if
$
\int_0^T (\Xi^\theta_t)^2\, \d\langle M \rangle_t \to \infty
$
almost surely, then the MLE is strongly consistent.
\end{pro}

\begin{proof}
By integrating the Langevin equation \eqref{eq:langevin} against the kernel $k^*$ on both sides, we obtain
$$
M^\theta_t = -\theta \int_0^t k^*(t,s) X^\theta_s\, \d s + M_t. 
$$
By plugging in $\Xi^\theta$, this translates into
$$
\d M^\theta_t = -\theta \Xi^\theta_t\, \d\langle M \rangle_t + \d M_t.
$$
Consequently, we can use the Girsanov theorem for Gaussian martingales, which states that the log-likelihood $\ell_T(\theta)=\log \d\P^\theta_T/\d\P_T$ can be written as  
\begin{eqnarray*}
\ell_T(\theta) 
= -\theta\int_0^T \Xi_t\, \d M^\theta_t - \frac{\theta^2}{2}\int_0^T (\Xi^\theta_t)^2\,\d\langle M \rangle_t.
\end{eqnarray*}
The MLE $\bar\theta_T$ follows from this by maximizing with respect to $\theta$. 

The strong consistency follows from the equation 
$$
\bar\theta_T - \theta
=
- \frac{\int_0^T \Xi^\theta_t\, \d M_t}{\int_0^T (\Xi^\theta_t)^2\, \d\langle M\rangle_t} 
$$
by using the martingale convergence theorem.
\end{proof}

%%%%%%%%%%%%%%%%%%%%%%%%%%%%%%%%%%%%%%%%%%%%%%%%%%%%%%%%%%%%%%%%%%%%%%%%%%%%%%%
%%%%%%%%%%%%%%%%%%%%%%%%%%%%%%%%%%%%%%%%%%%%%%%%%%%%%%%%%%%%%%%%%%%%%%%%%%%%%%%
\section{Conclusions}\label{sect:conclusions}

We have considered the Langevin equation with general stationary-increment continuous centered Gaussian noise.  We have stated mild conditions on the variance function of the noise ensuring strong consistency and asymptotic normality of the so-called alternative estimator of the mean-reversion parameter.  We have also provided Berry--Esseen bounds for the normal approximation.  We have shown that the alternative estimator works for discrete observations provided that the noise is H\"older continuous and the observation-mesh is dense enough with respect to the H\"older index of the noise. We have also shown that our results work in examples rising from fractional and bifractional Brownian noises, thus extending some recent results.  Finally, we discussed least squares estimators and maximum likelihood estimators. We argue that the alternative estimator is, under the stationarity assumption, the most general estimator, i.e., it works under the mildest assumptions.

%%%%%%%%%%%%%%%%%%%%%%%%%%%%%%%%%%%%%%%%%%%%%%%%%%%%%%%%%%%%%%%%%%%%%%%%%%%%%%%
%%%%%%%%%%%%%%%%%%%%%%%%%%%%%%%%%%%%%%%%%%%%%%%%%%%%%%%%%%%%%%%%%%%%%%%%%%%%%%%
%%%%%%%%%%%%%%%%%%%%%%%%%%%%%%%%%%%%%%%%%%%%%%%%%%%%%%%%%%%%%%%%%%%%%%%%%%%%%%%
\appendix
\section{Proofs of Lemmas}\label{apx}

\begin{proof}[Proof of Lemma \ref{lmm:v-growth}]
Let $t>0$ and let $\lfloor t\rfloor$ be the greatest integer not exceeding $t$. Then
$$
|G_t| \le |G_t-G_{\lfloor t \rfloor}| + \sum_{k=1}^{\lfloor t \rfloor} | G_k-G_{k-1}|.
$$
By the Minkowski's inequality and stationary of the increments,
$$
\sqrt{\E[G_t^2]} \le \sqrt{\E[G_{t-\lfloor t\rfloor}^2]}
+ \lfloor t \rfloor \sqrt{\E[G_1^2]}.
$$
The claim follows from this.
\end{proof}

\begin{proof}[Proof of Lemma \ref{lmm:invertibility}]
By changing the variable in \eqref{eq:rzero} we obtain
\begin{equation}\label{eq:psi}
\psi(\theta) = \frac{1}{2}\int_0^\infty \e^{-t}\,v\left(\frac{t}{\theta}\right)\, \d t.
\end{equation}
Since $v$ is strictly increasing, this shows that $\psi$ is also strictly decreasing.  Furthermore, $\psi(\theta)\to 0$ as $\theta\to\infty$ by the monotone convergence theorem. 

By the Lebesgue's dominated convergence theorem, the function
$$
\theta \mapsto \int_0^\infty \e^{-\theta t} v(t)\, \d t
$$
is smooth. Consequently, $\psi$ is smooth. 

Finally, let us show that $\psi$ is convex. Let us first assume that $v$ is differentiable.  Then, by applying the Lebesgue's dominated convergence to the representation \eqref{eq:psi} together with the change-of-variable $s=t/\theta$ we obtain
\begin{eqnarray*}
\psi'(\theta) &=& -\frac12 \int_0^\infty \e^{-t}\,\frac{t}{\theta^2}v'\left(\frac{t}{\theta}\right)\, \d t \\
&=& -\frac12\int_0^\infty \e^{-\theta s}\, s\,v'(s)\, \d s.
\end{eqnarray*}
Differentiating again we obtain, by using the Lebesgue's dominated convergence theorem,
$$
\psi''(\theta) =
\frac{1}{2} \int_0^\infty \e^{-\theta s} \, s^2v'(s)\, \d s \ge 0. 
$$
Consequently, $\psi$ is convex if $v$ is differentiable.  To conclude, the general case follows by approximating the continuous increasing function $v$ by differentiable increasing functions $v_n$, $n\in\mathbb{N}$, from  below.
\end{proof}

\begin{proof}[Proof of Lemma \ref{lmm:second-chaos}]
Let $t\ge0$ be fixed. Then we can represent $X_t = \sqrt{\gamma(t,t)}W_1$, where $W$ is a Brownian motion. (Here the Brownian motion depends on $t$ a priori. To see how to represent the entire process $X$ on any compact interval with a single Brownian motion see \cite[Theorem 3.1]{Sottinen-Viitasaari-2014-preprint}.) Consequently $X_t$ belongs to the $1^\mathrm{st}$ Wiener chaos for all $t>0$. Then note that 
$$
X_t^2 - \E[X_t^2]  = 2H_2(X_t).
$$
Consequently, it belongs to the $2^\mathrm{nd}$ Wiener chaos.  Finally, note that, because $\gamma$ is continuous, the integral
$$
\frac{1}{T}\int_0^T 2H_2(X_t)\, \d t 
$$
can be defined as a limit in $L^2(\Omega)$ in the $2^{\mathrm{nd}}$ Wiener chaos.  The claim follows from this.
\end{proof}

\begin{proof}[Proof of Lemma \ref{lmm:moments}]
%Let us begin with recalling the Isserlis' theorem in our case:
%\begin{equation}\label{eq:isserlis}
%\E\left[X_{t_1}\cdots X_{t_{2n}}\right] =
%\sum\prod\, \gamma(t_i,t_j),
%\end{equation}
%where $\sum\prod$ means summing over all distinct ways of partitioning $X_1,\ldots, X_{2n}$ into pairs. For odd number of terms the product is obviously zero. For orders two and four, we have 
%\begin{eqnarray*}
%\E\left[X_{t_1}X_{t_2}\right] 
%&=& \gamma(t_1,t_2), \\
%\E\left[X_{t_1}X_{t_2}X_{t_3}X_{t_4}\right] 
%&=&
%\gamma(t_1,t_2) + \gamma(t_2,t_3) + \gamma(t_3,t_4) + \gamma(t_4,t_1).
%\end{eqnarray*}
The claim follows exactly like in \cite[Lemma 2.2]{Viitasaari-2014-preprint} once the sums are replaced by integrals.
\end{proof}
\begin{proof}[Proof of Lemma \ref{lmm:inequality}]
Let $\Gamma$ be the operator associated with $\gamma$ by
$$
\Gamma f(t) = \int_{[0,T]} f(t_1)\,\gamma(t,t_1)\, \d t_1.
$$
Since $\gamma$ is a continuous covariance function, the operator $\Gamma$ is trace-class. Consequently, the kernel $\gamma$ admits the Mercer's expansion
$$
\gamma(t_1,t_2) = \sum_{n=1}^\infty \lambda_n \phi_n(t_1)\phi_n(t_2)
$$
with real eigenvalues $\lambda_n$ and continuous orthonormal eigenfunctions $\phi_n$. Let $\delta_{k,\ell}$ denote the Kronecker delta. Then, straightforward calculations with the Mercer's expansion by using the orthonormality of the $\phi_n$'s yield
\begin{eqnarray*}
\lefteqn{
\int_{[0,T]^4} \gamma(t_1,t_2)\gamma(t_2,t_3)\gamma(t_3,t_4)\gamma(t_4,t_1)\, \d t_1\d t_2 \d t_3 \d t_4} \\
&=&
\sum_{n_1,n_2,n_3,n_4=1}^\infty\int_{[0,T]^4} 
\lambda_{n_1}\phi_{n_1}(t_1)\phi_{n_1}(t_2) 
\lambda_{n_2}\phi_{n_2}(t_2)\phi_{n_2}(t_3) \\
& &
\lambda_{n_3}\phi_{n_3}(t_3)\phi_{n_3}(t_4)
\lambda_{n_4}\phi_{n_4}(t_4)\phi_{n_4}(t_1) 
\, \d t_1\d t_2 \d t_3 \d t_4\\
%&=&
%\sum_{n_1,n_2,n_3,n_4=1}^\infty \lambda_{n_1}\lambda_{n_2}\lambda_{n_3}\lambda_{n_4} 
%\int_0^T \phi_{n_1}(t_1)\phi_{n_4}(t_1)\, \d t_1 \\
%& &
%\int_0^T \phi_{n_1}(t_2)\phi_{n_2}(t_2)\, \d t_2
%\int_0^T \phi_{n_2}(t_3)\phi_{n_3}(t_3)\, \d t_3 \\
%& &
%\int_0^T \phi_{n_3}(t_4)\phi_{n_4}(t_4)\, \d t_4 \\
&=&
\sum_{n_1,n_2,n_3,n_4=1}^\infty \lambda_{n_1}\lambda_{n_2}\lambda_{n_3}\lambda_{n_4} \delta_{n_1,n_4}\delta_{n_1,n_2}\delta_{n_2,n_3}\delta_{n_3,n_4} \\
&=& \sum_{n=1}^\infty \lambda_n^4.
\end{eqnarray*}
Similar straightforward calculations also yield
\begin{eqnarray*}
\lefteqn{\int_{[0,T]^2} \gamma(t_1,t_2)^2\, \d t_1\d t_2} \\
%&=&\int_{[0,T]^2} \left(\sum_{n=1}^\infty \lambda_n \phi_n(t_1)\phi_n(t_2)\right)^2 \, \d t_1 \d t_2 \\ 
&=&
\sum_{n_1,n_2=1}^\infty\int_{[0,T]^2}  \lambda_{n_1}\lambda_{n_2} \phi_{n_1}(t_1)\phi_{n_2}(t_1)\phi_{n_1}(t_2)\phi_{n_2}(t_2) \, \d t_1 \d t_2 \\
%&=&
%\sum_{n_1,n_2=1}^\infty \lambda_{n_1}\lambda_{n_2} \int_0^T \phi_{n_1}(t_1)\phi_{n_2}(t_1)\, \d t_1 \int_0^T \phi_{n_1}(t_2)\phi_{n_2}(t_2)\, \d t_2 \\
&=&
\sum_{n_1,n_2=1}^\infty \lambda_{n_1}\lambda_{n_2}\delta_{n_1,n_2} \\
&=& \sum_{n=1}^\infty \lambda_n^2.
\end{eqnarray*}

Now, the operator $\Gamma$, being trace-class, admits maximal eigenvalue $\lambda^* = \max_n |\lambda_n|$.
Consequently, by using the elementary bound
$$
\sum_{n=1}^\infty \lambda_n^4 \le (\lambda^*)^2 \sum_{n=1}^\infty \lambda_n^2,
$$
the claim follows, if we can show that 
$$
\lambda^* \le \max_{t\in [0,T]} \int_0^T |\gamma(t,s)|\, \d s. 
$$
To show this, let $\phi^*$ be the eigenfunction corresponding to the maximal eigenvalue $\lambda^*$ and let $f^* = \phi^*/\|\phi^*\|_\infty$. Then
\begin{eqnarray*}
\|\Gamma\|_{\infty} &\ge& \|\Gamma f^* \|_\infty \\
&=&
\left\|\Gamma\left[\frac{\phi^*}{\|\phi^*\|_\infty}\right]\right\|_\infty \\
&=&
\left\|\frac{\lambda^*}{\|\phi^*\|_\infty}\phi^*\right\|_\infty \\
&=&
\lambda^* 
\end{eqnarray*}
Finally, note that 
\begin{eqnarray*}
\|\Gamma\|_{\infty} &=& \sup_{\|f\|_\infty=1} \| \Gamma f\|_{\infty} \\
&=& \sup_{\|f\|_\infty=1} \sup_{t\in [0,T]} \left|\int_0^T f(t_1)\gamma(t,t_1) \, \d t_1\right| \\
&\le& \sup_{\|f\|_\infty=1} \sup_{t\in [0,T]}\int_0^T \left|f(t_1)\right|\left|\gamma(t,t_1)\right|\, \d t_1.
\end{eqnarray*}
Now, the maximizing element in the sphere $\{ f \, ;\, \|f\|_\infty = 1\}$ is simply the constant function $f\equiv 1$, and the claim follows.
\end{proof}

\begin{proof}[Proof of Lemma \ref{lmm:Q-fourth}]
Note that $Q_T/\sqrt{\E[Q_T^2]}$ belongs to the $2^{\mathrm{nd}}$ Wiener chaos and has unit variance. Consequently, by the fourth-moment Proposition \ref{pro:fourth-berry-esseen}, it suffices to show that 
$$
\frac{\E\left[Q_T^4\right]}{\E\left[Q_T^2\right]^2} - 3
\le C \frac{\left(\sup_{t\in [0,T]} \int_{[0,T]} |\gamma(t,s)|\, \d s\right)^2}{\int_{[0,T]^2} \gamma(t,s)^2\, \d s\d t}.
$$
Denote
\begin{eqnarray*}
I_2(T) &=& \int_{[0,T]^2} \gamma(t_1,t_s)^2\, \d t_1\d t_2, \\
I_4(T) &=& \int_{[0,T]^4} \gamma(t_1,t_2)\gamma(t_2,t_3)\gamma(t_3,t_4)\gamma(t_4,t_1)\, \d t_1 \d t_2 \d t_3 \d t_4.
\end{eqnarray*}
Then, by Lemma \ref{lmm:moments},
\begin{eqnarray*}
\E[Q_T^4] &=& \frac{12}{T^4} I_2(T)^2 + \frac{24}{T^4} I_4(T), \\
\E[Q_T^2]^2 &=& \frac{4}{T^4} I_2(T)^2,
\end{eqnarray*}
and, by Lemma \ref{lmm:inequality}
$$
I_4(T) \le \left[\sup_{t\in [0,T]}\int_0^T |\gamma(t,s)|\, \d s\right]^2 I_2(T).
$$
Consequently,
\begin{eqnarray*}
\frac{\E\left[Q_T^4\right]}{\left(\E\left[Q_T^2\right]\right)^2} -3
&=&
\frac{12 I_2(T)^2+24 I_4(T)-12I_2(T)^2}{4 I_2(T)^2} \\
&=& 
6\frac{I_4(T)}{I_2(T)^2} \\
&\le&
6\frac{\left[\sup_{t\in [0,T]}\int_0^T |\gamma(t,s)|\, \d s\right]^2}{I_2(T)}.
\end{eqnarray*}
The claim follows from this.
\end{proof}

\begin{proof}[Proof of Lemma \ref{lmm:asy-var}]
By Lemma \ref{lmm:moments}, Proposition \ref{pro:gamma-vs-r}, symmetry and estimate $|r_\theta(t)|\le r_\theta(0)$,
\begin{eqnarray*}
\lefteqn{\left|\E[(Q^\theta_T)^2] - w_\theta(T)\right|} \\ 
&\le&
\frac{4}{T^2}\int_0^T\!\!\!\!\int_0^t |r_\theta(t-s)|k_\theta(t,s)\, \d s \d t
+ \frac{2}{T^2}\int_0^T\!\!\!\!\int_0^t k_\theta(t,s)^2\, \d s \d t \\
&\le& 
\frac{C_\theta}{T^2}\int_0^T\!\!\!\!\int_0^t r_\theta(0)\e^{-\theta s}\, \d s \d t
+ \frac{C_\theta}{T^2}\int_0^T\!\!\!\!\int_0^t \e^{-2\theta s} \, \d s \d t \\
&=&
\frac{C_\theta}{T^2}\int_0^T\left[1-\e^{-\theta t} \right] \d t
+ \frac{C_\theta}{T^2}\int_0^T\left[1-\e^{-2\theta t}\right] \d t \\
&\le&
\frac{C_\theta}{T}.
\end{eqnarray*}
This shows the estimate.
\end{proof}

\begin{proof}[Proof of Lemma \ref{lmm:equiv-var}]
The equivalence $\E[(Q^\theta_T)^2] \sim w_\theta(T)$ follows from Lemma \ref{lmm:asy-var}, if $Tw_\theta(T)$ does not converge to zero.  Now, by symmetry, change-of-variables and the Fubini theorem
$$
\int_0^T\!\!\!\!\int_0^T r_\theta(t-s)^2\, \d s \d t
=
2\int_0^T\!\!\!\!\int_{t}^T r_\theta(s)^2 \, \d t \d s \\
=
2\int_0^T r_\theta(t)^2(T-t)\, \d t.
$$
Consequently, by assuming that $T> 1$,
$$
\int_0^T r_\theta(t)^2(T-t)\, \d t
\ge
\int_0^1 r_\theta(t)^2 (T-t)\, \d t 
\ge C(T-1)
$$
This shows $T w_\theta (T)\ge C$. Also, we have shown the equality
$$
w_\theta(T) = \frac{4}{T}\int_0^T r_\theta(t)^2(T-t)\, \d t
$$

Consider then the best-rate case $\int_0^\infty r_\theta(t)^2\, \d t < \infty$. By the equivalence $\E[(Q^\theta_T)^2] \sim \frac{4}{T}\int_0^T r_\theta(t)^2(T-t)\, \d t$, it is enough to show that
$$
\lim_{T\to\infty}\frac{\int_T^\infty r_\theta(t)^2\, \d t}{\int_0^T r(t)^2(T-t)\, \d t} 
= 0.
$$
But this is immediate from the facts that the nominator converges to zero and the denominator is bounded from below by $C(T-1)$.
\end{proof}

\begin{proof}[Proof of Lemma \ref{lmm:Q-conv}] 
Let us first split
\begin{eqnarray*}
\lefteqn{\sup_{x\in\R}\left|\P\left[\frac{Q^\theta_T}{\sqrt{w_\theta(T)}}\le x\right]-\Phi(x)\right|} \\
&\le&
\sup_{x\in\R}\left|\P\left[\frac{Q^\theta_T}{\sqrt{\E[(Q^\theta_T)^2]}} 
\le \sqrt{\frac{w_\theta(T)}{\E[(Q^\theta_T)^2]}} x\right]
-\Phi\left(\sqrt{\frac{w_\theta(T)}{\E[(Q^\theta_T)^2]}} x\right)\right|
\\
& &+ \sup_{x\in\R}\left|
\P\left[\Phi\left(\sqrt{\frac{w_\theta(T)}{\E[(Q^\theta_T)^2]}} x\right)\right]-\Phi(x)
\right| \\
&\le&
\sup_{x\in\R}\left|\P\left[\frac{Q^\theta_T}{\sqrt{\E[(Q^\theta_T)^2]}} 
\le x\right]
-\Phi\left(x\right)\right| 
\\ & &
+ \sup_{x\in\R}\left|
\Phi\left(\sqrt{\frac{w_\theta(T)}{\E[(Q^\theta_T)^2]}} x\right)-\Phi(x)
\right|\\
&=& A_1 + A_2. 
\end{eqnarray*}

For the term $A_1$, let us first estimate
\begin{eqnarray*}
\lefteqn{\int_0^T |\gamma_\theta(t,s)|\, \d s } \\
&=& 
\int_0^T \left|
r_\theta(t-s) + \e^{-\theta(t+s)}r_\theta(0) 
- \e^{-\theta t}r_\theta(s) - \e^{-\theta s}r_\theta(t)
\right|\, \d s \\
&\le&
\int_0^T |r_\theta(t-s)|\, \d s %\\ & & 
+ \frac{\e^{-\theta t}|r_\theta(0)|}{\theta}%\left[1-\e^{-\theta T}\right]
+ \e^{-\theta t}\int_0^T |r_\theta(s)|\, \d s
+ \frac{|r_\theta(t)|}{\theta}%\left[1-\e^{-\theta T}\right] 
\\
&\le&
\int_0^T |r_\theta(t-s)|\, \d s 
+ \e^{-\theta t}\int_0^T |r_\theta(s)|\, \d s 
+ \frac{\left(\e^{-\theta t}+1\right)|r_\theta(0)|}{\theta}
%\left[1-\e^{-\theta T}\right]
\\
&\le&
\int_0^T |r_\theta(t-s)|\, \d s + \int_0^T|r_\theta(s)| \, \d s + C,
\end{eqnarray*}
Consequently,
\begin{eqnarray*}
\sup_{t\in [0,T]} \int_0^T |\gamma_\theta(t,s)|\, \d s
&\le& 2 \sup_{t\in [0,T]} \int_0^T |r_\theta(t-s)|\, \d s +C\\
&=& 2\sup_{t\in[0,T]}\int_{-t}^{T-t} |r_\theta(u)|\, \d u +C \\
&\le& 2\int_{-T}^{T} |r_\theta(u)|\, \d u + C \\
&=& 4\int_0^T |r_\theta(u)|\, \d u + C.
\end{eqnarray*}
Since we are interested in the case $T\to\infty$, we can assume that $T$ is bigger than some absolute positive constant.  Consequently, since $r_\theta$ continuous with $r_\theta(0)>0$, it follows from the estimate above that
$$
\sup_{t\in [0,T]} \int_0^T |\gamma_\theta(t,s)|\, \d s \le
C\int_0^T |r_\theta(u)| \, \d u.
$$
Therefore, by applying Lemma \ref{lmm:Q-fourth} and Lemma \ref{lmm:equiv-var}, it follows that $A_1\le C_\theta R_\theta(T)$.

Let us then consider the term $A_2$.  Now, by the mean value theorem,
\begin{eqnarray*}
A_2 &=&
\sup_{x\in\R}\left|\Phi\left(\sqrt{\frac{w_\theta(T)}{\E[(Q^\theta_T)^2]}} x\right)-\Phi(x)\right| \\
&\le&
\frac{1}{\sqrt{2\pi}}\sup_{x\in\R}\left( \e^{-1/2 \eta_\theta(T,x)^2}|x|\right)\,
\left|\sqrt{\frac{w_\theta(T,x)}{\E[(Q^\theta_T)^2]}}-1\right|,
\end{eqnarray*}
where 
$$
\eta_\theta(T,x) \in \left[x,x+\sqrt{\frac{w_\theta(T)}{\E[(Q^\theta_T)^2]}}x\right].
$$
Since $\sqrt{\frac{w_\theta(T)}{\E[(Q^\theta_T)^2]}} \sim 1$, it follows that
$$
A_2 \le \left|\sqrt{\frac{w_\theta(T)}{\E\left[(Q^\theta_T)^2\right]}}-1\right|
=
\frac{\left|\sqrt{w_\theta(T,x)}-\sqrt{\E\left[(Q^\theta_T)^2\right]}\right|}{\sqrt{\E\left[(Q^\theta_T)^2\right]}}.
$$
Consequently, by the asymptotic equivalence of $w_\theta(T)\sim \E[(Q^\theta_T)^2]$, it remains to show that
$$
T\left|\sqrt{w_\theta(T)}-\sqrt{\E[(Q^\theta_T)^2]}\right|
\le C_\theta \int_0^T |r_\theta(t)|\, \d t.
$$
(Actually, we show that the left hand side is bounded.)
For this purpose, we estimate, by using the inequality $|\sqrt{a}-\sqrt{b}| \le \sqrt{|a-b|}$ and the identity $a^2-b^2=(a+b)(a-b)$, that
\begin{eqnarray*}
\lefteqn{T\left|\sqrt{w_\theta(T)}-\sqrt{\E[(Q^\theta_T)^2]}\right|} \\
%&=&
%T\left|\sqrt{\frac{2}{T^2}\int_0^T\!\!\!\!\int_0^T r_\theta(t-s)^2\, \d s \d t}-
%\sqrt{\frac{1}{T^2}\int_0^T\!\!\!\!\int_0^T \gamma_\theta(t,s)^2\, \d s \d t}\right|\\
&=& \sqrt{2}
\left|\sqrt{\int_0^T\!\!\!\!\int_0^T r_\theta(t-s)^2\, \d s \d t}-
\sqrt{\int_0^T\!\!\!\!\int_0^T \gamma_\theta(t,s)^2\, \d s \d t}\right| \\
%&\le& \sqrt{\left|2\int_0^T\!\!\!\!\int_0^T r_\theta(t-s)^2\, \d s \d t-
%\int_0^T\!\!\!\!\int_0^T \gamma_\theta(t,s)^2\, \d s \d t\right|} \\
&\le& \sqrt{2}
\sqrt{\left|\int_0^T\!\!\!\!\int_0^T \left[ r_\theta(t-s)^2- \gamma_\theta(t,s)^2\, \right] \d s \d t\right|}\\
&=& \sqrt{2}
\sqrt{\left|\int_0^T\!\!\!\!\int_0^T \big(r_\theta(t-s)+\gamma_\theta(t,s)\big)\big(r_\theta(t-s)-\gamma_\theta(t,s)\big) \d s \d t\right|}.
\end{eqnarray*}
By applying Proposition \ref{pro:gamma-vs-r} to the estimate above, we obtain
\begin{eqnarray*}
\lefteqn{T\left|\sqrt{w_\theta(T)}-\sqrt{\E[(Q^\theta_T)^2]}\right|} \\
&\le&
C_\theta\sqrt{\left|\int_0^T\!\!\!\!\int_0^T \big(r_\theta(t-s)+\gamma_\theta(t,s)\big)\e^{-\theta\min(s,t)}\, \d s \d t\right|}.
\end{eqnarray*}
Now, $|r_\theta(t-s)|\le r_\theta(0)$ and $|\gamma_\theta(t,s)|\le r_\theta(0) +1$, by Proposition \ref{pro:gamma-vs-r}.  Consequently, the integral above is bounded, and the proof is finished. 
\end{proof}

\begin{proof}[Proof of Lemma \ref{lmm:bifractional}]
Let 
$$
a(t) = a_{H,1}(t) = H \e^{t/H}.
$$
Then
$$
r_{H,K,\theta}(t) =
\frac{1}{2^K}\e^{-\theta t}\left[\big(a(t)^{2H}+1\big)^K - \big(a(t)-1\big)^{2HK}\right].
$$
By the Taylor's theorem
\begin{eqnarray*}
\big(a(t)^{2H}+1\big)^{K} &=&
a(t)^{2HK} + K \xi(t)^{K-1}, \\
\big(a(t)-1\big)^{2HK} &=&
a(t)^{2HK} - 2HK\eta(t)^{2HK-1},
\end{eqnarray*}
for some $\xi(t) \in [a(t), a(t)+1]$ and  $\eta(t)\in [a(t)-1,a(t)]$. Consequently,
\begin{eqnarray*}
r_{H,K,\theta}(t) &\sim&
C_{H,K,\theta} \e^{-\theta t}\left[
a(t)^{K-1} + a(t)^{2HK-1} 
\right] \\
&\sim&
C_{H,K,\theta} \e^{-\theta t \max\left\{\frac{1}{HK}-1,1+\frac{1}{HK}-\frac{1}{H}\right\}},
\end{eqnarray*}
which shows the exponential decay.
\end{proof}

%%%%%%%%%%%%%%%%%%%%%%%%%%%%%%%%%%%%%%%%%%%%%%%%%%%%%%%%%%%%%%%%%%%%%%%%%%%%%%%
\bibliographystyle{siam}
\bibliography{../pipliateekki}

\begin{thebibliography}{10}

\bibitem{Azmoodeh-Morlanes-2015}
{\sc E.~Azmoodeh and J.~I. Morlanes}, {\em Drift parameter estimation for
  fractional {O}rnstein-{U}hlenbeck process of the second kind}, Statistics, 49
  (2015), pp.~1--18.

\bibitem{Azmoodeh-Sottinen-Viitasaari-Yazigi-2014}
{\sc E.~Azmoodeh, T.~Sottinen, L.~Viitasaari, and A.~Yazigi}, {\em Necessary
  and sufficient conditions for {H}\"older continuity of {G}aussian processes},
  Statist. Probab. Lett., 94 (2014), pp.~230--235.

\bibitem{Azmoodeh-Viitasaari-2015a}
{\sc E.~Azmoodeh and L.~Viitasaari}, {\em Parameter estimation based on
  discrete observations of fractional {O}rnstein-{U}hlenbeck process of the
  second kind}, Stat. Inference Stoch. Process., 18 (2015), pp.~205--227.

\bibitem{Barndorff-Nielsen-Basse-OConnor-2011}
{\sc O.~E. Barndorff-Nielsen and A.~Basse-O'Connor}, {\em Quasi
  {O}rnstein-{U}hlenbeck processes}, Bernoulli, 17 (2011), pp.~916--941.

\bibitem{Belyaev-1960}
{\sc Y.~{Belyaev}}, {\em {Local properties of the sample functions of
  stationary Gaussian processes.}}, {Teor. Veroyatn. Primen.}, 5 (1960),
  pp.~128--131.

\bibitem{Biagini-Hu-Oksendal-Zhang-2008}
{\sc F.~Biagini, Y.~Hu, B.~{\O}ksendal, and T.~Zhang}, {\em Stochastic calculus
  for fractional {B}rownian motion and applications}, Probability and its
  Applications (New York), Springer-Verlag London, Ltd., London, 2008.

\bibitem{Breton-Nourdin-2008}
{\sc J.-C. Breton and I.~Nourdin}, {\em Error bounds on the non-normal
  approximation of {H}ermite power variations of fractional {B}rownian motion},
  Electron. Commun. Probab., 13 (2008), pp.~482--493.

\bibitem{Cheridito-Kawaguchi-Maejima-2003}
{\sc P.~Cheridito, H.~Kawaguchi, and M.~Maejima}, {\em Fractional
  {O}rnstein-{U}hlenbeck processes}, Electron. J. Probab., 8 (2003), pp.~no. 3,
  14 pp. (electronic).

\bibitem{Doob-1942}
{\sc J.~L. Doob}, {\em The {B}rownian movement and stochastic equations}, Ann.
  of Math. (2), 43 (1942), pp.~351--369.

\bibitem{Es-Sebaiy-Ndiaye-2014}
{\sc K.~Es-Sebaiy and D.~Ndiaye}, {\em On drift estimation for non-ergodic
  fractional {O}rnstein-{U}hlenbeck process with discrete observations}, Afr.
  Stat., 9 (2014), pp.~615--625.

\bibitem{Es-Sebaiy-Tudor-2015}
{\sc K.~Es-Sebaiy and C.~A. Tudor}, {\em Fractional {O}rnstein-{U}hlenbeck
  processes mixed with a gamma distribution}, Fractals, 23 (2015), pp.~1550032,
  10.

\bibitem{Gauss-1809}
{\sc C.~F. Gauss}, {\em Theoria motus corporum coelestium in sectionibus
  conicis solem ambientium}, Cambridge Library Collection, Cambridge University
  Press, Cambridge, 2011.
\newblock Reprint of the 1809 original.

\bibitem{Grenander-1950}
{\sc U.~Grenander}, {\em Stochastic processes and statistical inference}, Ark.
  Mat., 1 (1950), pp.~195--277.

\bibitem{Houdre-Villa-2003}
{\sc C.~Houdr{\'e} and J.~Villa}, {\em An example of infinite dimensional
  quasi-helix}, in Stochastic models ({M}exico {C}ity, 2002), vol.~336 of
  Contemp. Math., Amer. Math. Soc., Providence, RI, 2003, pp.~195--201.

\bibitem{Hu-Nualart-2010}
{\sc Y.~Hu and D.~Nualart}, {\em Parameter estimation for fractional
  {O}rnstein-{U}hlenbeck processes}, Statist. Probab. Lett., 80 (2010),
  pp.~1030--1038.

\bibitem{Hu-Song-2013}
{\sc Y.~Hu and J.~Song}, {\em Parameter estimation for fractional
  {O}rnstein-{U}hlenbeck processes with discrete observations}, in Malliavin
  calculus and stochastic analysis, vol.~34 of Springer Proc. Math. Stat.,
  Springer, New York, 2013, pp.~427--442.

\bibitem{Kaarakka-Salminen-2011}
{\sc T.~Kaarakka and P.~Salminen}, {\em On fractional {O}rnstein-{U}hlenbeck
  processes}, Commun. Stoch. Anal., 5 (2011), pp.~121--133.

\bibitem{Kleptsyna-LeBreton-2002}
{\sc M.~L. Kleptsyna and A.~Le~Breton}, {\em Statistical analysis of the
  fractional {O}rnstein-{U}hlenbeck type process}, Stat. Inference Stoch.
  Process., 5 (2002), pp.~229--248.

\bibitem{Kozachenko-Melnikov-Mishura-2015}
{\sc Y.~Kozachenko, A.~Melnikov, and Y.~Mishura}, {\em On drift parameter
  estimation in models with fractional {B}rownian motion}, Statistics, 49
  (2015), pp.~35--62.

\bibitem{Kubilius-Mishura-Ralchenko-Seleznjev-2015}
{\sc K.~Kubilius, Y.~Mishura, K.~Ralchenko, and O.~Seleznjev}, {\em Consistency
  of the drift parameter estimator for the discretized fractional
  {O}rnstein-{U}hlenbeck process with {H}urst index {$H\in(0,\frac{1}{2})$}},
  Electron. J. Stat., 9 (2015), pp.~1799--1825.

\bibitem{Lamperti-1962}
{\sc J.~Lamperti}, {\em Semi-stable stochastic processes}, Trans. Amer. Math.
  Soc., 104 (1962), pp.~62--78.

\bibitem{Langevin-1908}
{\sc P.~Langevin}, {\em {Sur la th\'{e}orie du mouvement brownien}}, C. R.
  Acad. Sci. Paris, 146 (1908), pp.~530--533.

\bibitem{Liptser-Shiryaev-II-2001}
{\sc R.~S. Liptser and A.~N. Shiryaev}, {\em Statistics of random processes.
  {II}}, vol.~6 of Applications of Mathematics (New York), Springer-Verlag,
  Berlin, expanded~ed., 2001.
\newblock Applications, Translated from the 1974 Russian original by A. B.
  Aries, Stochastic Modelling and Applied Probability.

\bibitem{Maruyama-1949}
{\sc G.~Maruyama}, {\em The harmonic analysis of stationary stochastic
  processes}, Mem. Fac. Sci. Ky\=usy\=u Univ. A., 4 (1949), pp.~45--106.

\bibitem{Mishura-2008}
{\sc Y.~S. Mishura}, {\em Stochastic calculus for fractional {B}rownian motion
  and related processes}, vol.~1929 of Lecture Notes in Mathematics,
  Springer-Verlag, Berlin, 2008.

\bibitem{Nualart-2006}
{\sc D.~Nualart}, {\em The {M}alliavin calculus and related topics},
  Probability and its Applications (New York), Springer-Verlag, Berlin,
  second~ed., 2006.

\bibitem{Peccati-Taqqu-2011}
{\sc G.~Peccati and M.~S. Taqqu}, {\em Wiener chaos: moments, cumulants and
  diagrams}, vol.~1 of Bocconi \& Springer Series, Springer, Milan; Bocconi
  University Press, Milan, 2011.
\newblock A survey with computer implementation, Supplementary material
  available online.

\bibitem{Russo-Tudor-2006}
{\sc F.~Russo and C.~A. Tudor}, {\em On bifractional {B}rownian motion},
  Stochastic Process. Appl., 116 (2006), pp.~830--856.

\bibitem{Shen-Yin-Yan-2016}
{\sc G.~Shen, X.~Yin, and L.~Yan}, {\em Least squares estimation for
  {O}rnstein--{U}hlenbeck processes driven by the weighted fractional
  {B}rownian motion}, Acta Math. Sci. Ser. B Engl. Ed., 36 (2016),
  pp.~394--408.

\bibitem{Shen-Xu-2014}
{\sc L.~Shen and Q.~Xu}, {\em Asymptotic law of limit distribution for
  fractional {O}rnstein-{U}hlenbeck process}, Adv. Difference Equ.,  (2014),
  pp.~2014:75, 7.

\bibitem{Sottinen-Tudor-2008}
{\sc T.~Sottinen and C.~A. Tudor}, {\em Parameter estimation for stochastic
  equations with additive fractional {B}rownian sheet}, Stat. Inference Stoch.
  Process., 11 (2008), pp.~221--236.

\bibitem{Sottinen-Viitasaari-2014-preprint}
{\sc T.~Sottinen and L.~Viitasaari}, {\em Stochastic analysis of {G}aussian
  processes via {F}redholm representation}, submitted, arXiv: 1410.2230 (2014).

\bibitem{Sottinen-Yazigi-2014}
{\sc T.~Sottinen and A.~Yazigi}, {\em Generalized {G}aussian bridges},
  Stochastic Process. Appl., 124 (2014), pp.~3084--3105.

\bibitem{Sun-Guo-2015}
{\sc X.~Sun and F.~Guo}, {\em On integration by parts formula and
  characterization of fractional {O}rnstein--{U}hlenbeck process}, Statist.
  Probab. Lett., 107 (2015), pp.~170--177.

\bibitem{Tanaka-2015}
{\sc K.~Tanaka}, {\em Maximum likelihood estimation for the non-ergodic
  fractional {O}rnstein-{U}hlenbeck process}, Stat. Inference Stoch. Process.,
  18 (2015), pp.~315--332.

\bibitem{Ornstein-Uhlenbeck-1930}
{\sc G.~E. Uhlenbeck and L.~S. Ornstein}, {\em On the theory of the brownian
  motion}, Phys. Rev., 36 (1930), pp.~823--841.

\bibitem{Viitasaari-2014-preprint}
{\sc L.~{Viitasaari}}, {\em {Representation of stationary and stationary
  increment processes via Langevin equation and self-similar processes}}, ArXiv
  e-prints,  (2014).

\bibitem{Yazigi-2015}
{\sc A.~Yazigi}, {\em Representation of self-similar {G}aussian processes},
  Statist. Probab. Lett., 99 (2015), pp.~94--100.

\end{thebibliography}
\end{document}